\documentclass[11pt]{article}
\usepackage{amsmath, amsfonts, amssymb}

\newcommand{\PSPACE}{\mathsf{PSPACE}}
\newcommand{\NP}{\mathsf{NP}}
\newcommand{\LFP}{\mathsf{LFP}}
\newcommand{\ESO}{\mathsf{ESO}}

\newcommand{\E}{\mathbb{E}}

\newcommand{\FOP}{\mathsf {FO}[\oplus]}
\newcommand{\FOM}{\mathsf {FO}[\Mod_q]}
\newcommand{\FO}{\mathsf {FO}}
\newcommand{\AC}{\mathsf {AC0}[\oplus]}
\newcommand{\Hom}{\mathsf {Hom}}
\newcommand{\Inj}{\mathsf {Inj}}
\newcommand{\K}{\mathrm {K}}

\newcommand{\Cop}{\mathsf {Cop}}

\newcommand{\aut}{\mathsf {aut}}
\newcommand{\Aut}{\mathsf {Aut}}
\newcommand{\pmat}{\mathsf {PMatch}}

\newcommand{\Conn}{\mathsf {Conn}}

\newcommand{\Types}{\mathsf {Types}}
\newcommand{\type}{\mathsf {type}}
\newcommand{\freq}{\mathsf {freq}}
\newcommand{\supp}{\mathsf {supp}}
\newcommand{\Mod}{\mathsf {Mod}}
\newcommand{\Freq}{\mathsf {FFreq}}
\newcommand{\konn}{label-connected }

\newcommand{\Partitions}{\mathsf {Partitions}}
\newcommand{\modsum}{\sum}

\newcommand{\handout}[5]{
   \noindent
   \begin{center}
   \framebox{
      \vbox{
    \hbox to 5.78in { {\bf 6.896 Sublinear Time Algorithms}
     	 \hfill #2 }
       \vspace{4mm}
       \hbox to 5.78in { {\Large \hfill #5  \hfill} }
       \vspace{2mm}
       \hbox to 5.78in { {\it #3 \hfill #4} }
      }
   }
   \end{center}
   \vspace*{4mm}
}

\newtheorem{theorem}{Theorem}[section]
\newtheorem{corollary}[theorem]{Corollary}
\newtheorem{lemma}[theorem]{Lemma}

\newtheorem{proposition}[theorem]{Proposition}
\newtheorem{definition}[theorem]{Definition}
\newtheorem{claim}[theorem]{Claim}

\newcommand{\qed}{\hfill \ensuremath{\Box}}

\newenvironment{proof}{\noindent{\bf Proof}\hspace*{1em}}{\qed\bigskip}
\newenvironment{proof-sketch}{\noindent{\bf Sketch of Proof}\hspace*{1em}}{\qed\bigskip}
\newenvironment{proof-idea}{\noindent{\bf Proof Idea}\hspace*{1em}}{\qed\bigskip}
\newenvironment{proof-of-lemma}[1]{\noindent{\bf Proof of Lemma #1}\hspace*{1em}}{\qed\bigskip}
\newenvironment{proof-attempt}{\noindent{\bf Proof Attempt}\hspace*{1em}}{\qed\bigskip}
\newenvironment{proofof}[1]{\noindent{\bf Proof
of #1:\hspace*{1em}}}{\qed\bigskip}
\newenvironment{remark}{\noindent{\bf Remark}\hspace*{1em}}{\bigskip}

\makeatletter
\def\fnum@figure{{\bf Figure \thefigure}}
\def\fnum@table{{\bf Table \thetable}}
\long\def\@mycaption#1[#2]#3{\addcontentsline{\csname
  ext@#1\endcsname}{#1}{\protect\numberline{\csname 
  the#1\endcsname}{\ignorespaces #2}}\par
  \begingroup
    \@parboxrestore
    \small
    \@makecaption{\csname fnum@#1\endcsname}{\ignorespaces #3}\par
  \endgroup}
\def\mycaption{\refstepcounter\@captype \@dblarg{\@mycaption\@captype}}
\makeatother

\newcommand{\mathify}[1]{\ifmmode{#1}\else\mbox{$#1$}\fi}
\newcommand{\bigO}O

\newcommand{\Z}{{\mathbb Z}}

\renewcommand{\vec}[1]{{\mathbf #1}}

\begin{document}
\begin{titlepage}
\title{Random Graphs and the Parity Quantifier}
\author{Phokion G.\ Kolaitis\thanks{UC Santa Cruz and IBM Almaden Research Center. {\tt kolaitis@cs.ucsc.edu}} \and
Swastik Kopparty\thanks{Massachusetts Institute of Technology. {\tt swastik@mit.edu}}}

\maketitle
\thispagestyle{empty}
\begin{abstract}
The classical zero-one law for first-order logic on random graphs says that for every first-order property $\varphi$ in the theory of graphs and every $p \in (0,1)$, the probability that the random graph $G(n, p)$ satisfies
$\varphi$ approaches either $0$ or $1$ as $n$ approaches infinity.
It is well known that this law fails to hold for
 any formalism that can express the parity quantifier:
 for certain properties, the probability that $G(n,p)$ satisfies the property need not converge, and for
others the limit may be strictly between $0$ and $1$.

In this work, we capture the limiting behavior
of properties definable in first order logic augmented with the parity quantifier, $\FOP$,
over $G(n,p)$, thus eluding the above hurdles. Specifically,
we establish the following ``modular convergence law":
\begin{quote}
For every $\FOP$ sentence $\varphi$, there are two explicitly computable rational numbers $a_0$, $a_1$, such that for $i \in \{0,1\}$, as $n$ approaches infinity, the probability that the random graph $G(2n+i, p)$ satisfies $\varphi$ approaches $a_i$.
\end{quote}
Our results also extend appropriately to $\FO$ equipped with $\Mod_q$ quantifiers for prime $q$.

In the process of deriving the above theorem, we explore a new question
that may be of interest in its own right. Specifically, we study the
joint distribution of the subgraph statistics modulo $2$ of $G(n,p)$:
namely, the number of copies, mod $2$, of a fixed number of graphs
$F_1, \ldots, F_\ell$ of bounded size in $G(n,p)$. We first
show that every $\FOP$ property $\varphi$ is almost surely determined by
subgraph statistics modulo $2$ of the above type.
Next, we show that the limiting joint distribution
of the subgraph statistics modulo $2$ depends
only on $n \mod 2$, and we determine
this limiting distribution completely. Interestingly,
both these steps are based on a common technique using
multivariate polynomials over finite fields and,
in particular, on a new generalization  of the Gowers norm.

The first step above is analogous to the Razborov-Smolensky method
for lower bounds for $\mathsf{AC0}$ with parity gates, yet stronger
in certain ways. For instance, it allows us to obtain examples of
simple graph properties that are exponentially uncorrelated with
every $\FOP$ sentence, which is something that is not known for $\AC$.
\end{abstract}
\end{titlepage}
\thispagestyle{empty}
\setcounter{page}{0}
\tableofcontents
\newpage
\section{Introduction}

For quite a long time, combinatorialists have studied the asymptotic probabilities
of properties on classes of finite structures, such as graphs and partial orders.
Assume that $\cal C$ is a class of finite structures and let $\Pr_n$, $n\geq 1$,
be a sequence of probability measures on all structures in $\cal C$ with  $n$ elements
in their domain. If $Q$ is a property of some structures in $\cal C$ (that is, a decision problem on $\cal C$),
then the \emph{asymptotic probability $\Pr(Q)$ of $Q$ on $\cal C$} is defined as $\Pr(Q)= \lim_{n\rightarrow \infty}\Pr_n(Q)$,
provided this limit exists.
In this paper, we will be focusing on the case when $\cal C$ is the class $\cal G$ of all finite graphs,
and $\Pr_n = G(n,p)$ for constant $p$; this is the probability distribution on $n$-vertex
undirected graphs where between each pair of nodes an edge appears with probability $p$, independently
of other pairs of nodes.
 For example, for this case,
the asymptotic probabilities $\Pr(\mbox{\sc{Connectivity}}) = 1$ and
 $\Pr(\mbox{{\sc Hamiltonicity}}) = 1$; in contrast,
if $\Pr_n=G(n, p(n))$ with $p(n)= 1/n$, then $\Pr(\mbox{\sc{Connectivity}}) = 0$ and $\Pr(\mbox{{\sc Hamiltonicity}}) = 0$.

Instead of studying separately one property at a time, it is natural to consider formalisms for
specifying properties of finite structures and to investigate the connection between the expressibility
of a property in a certain formalism and its asymptotic probability. The first and most celebrated such connection was established by
Glebskii et al.\ \cite{GKLT69} and, independently, by Fagin \cite{Fa76}, who showed that
 a \emph{0-1 law} holds for first-order logic\footnote{Recall that the formulas of first-order logic on graphs
 are obtained from atomic formulas $E(x,y)$ (interpreted as the adjacency relation) and  equality formulas $x=y$ using Boolean combinations,  existential
 quantification, and universal quantification;
  the quantifiers are interpreted as ranging over the set of vertices of the graph (and not over sets
  of vertices or sets of edges, etc.).}
 $\FO$ on the random graph $G(n,p)$ with $p$ a constant in $(0,1)$; this means that
 if $Q$ is a property of graphs expressible in  $\FO$ and $\Pr_n=G(n,p)$ with $p$ a constant in $(0,1)$, then $\Pr(Q)$ exists and is
either $0$ or $1$. This result became the catalyst for a series of investigations in several different directions.
Specifically, one line of investigation \cite{SS87,SS88} investigated  the existence of 0-1 laws for first-order logic $\FO$ on the random graph $G(n,p(n))$ with
$p(n) = n^{-\alpha}$, $0<\alpha<1$. Since first-order logic on finite graphs has limited
 expressive power (for example, $\FO$ cannot express  {\sc Connectivity} and
 {\sc 2-Colorability}), a different line of investigation pursued 0-1 laws for extensions
of first-order logic on the random graph $G(n,p)$ with $p$ a constant in $(0,1)$. In this vein,
it was shown in \cite{BGK85,KV87} that the 0-1 law holds for extensions of $\FO$ with fixed-point operators, such
as least fixed-point logic $\LFP$, which can express {\sc Connectivity} and {\sc 2-Colorability}.
As regards to higher-order logics,  it is clear that the 0-1 law fails even for existential second-order logic $\ESO$,
 since it is well known that $\ESO = \NP$ on finite graphs \cite{Fa74}. In fact, even the \emph{convergence law} fails
 for $\ESO$, that is, there are $\ESO$-expressible properties $Q$ of finite graphs such that $\Pr(Q)$ does \emph{not}
 exist.
 For this reason, a separate
line of investigation pursued 0-1 laws for syntactically-defined subclasses of $\NP$. Eventually, this
investigation produced a complete classification of the quantifier prefixes of $\ESO$ for which
the 0-1 law holds \cite{KV87,KV90c,PS89}, and provided a unifying account for the asymptotic
probabilities of such NP-complete problems as {\sc $k$-Colorability}, $k\geq 3$.

Let $L$ be a logic for which the 0-1 law (or even just the convergence law) holds on the random graph $G(n,p)$ with $p$ a constant
in $(0,1)$. An immediate consequence of this is that $L$ cannot express any \emph{counting} properties,
such as {\sc Even Cardinality} (``there is an even number of nodes"), since $\Pr_{2n}(\mbox{{\sc Even Cardinality}})=1$ and $\Pr_{2n+1}(\mbox{{\sc Even Cardinality}}) = 0$. In this paper, we turn the tables around
and systematically investigate the asymptotic probabilities of properties expressible in extensions of $\FO$ with  \emph{counting
quantifiers} $\Mod^i_q$, where $q$ is a prime number. The most prominent such extension is
$\FOP$, which is the extension of $\FO$ with the \emph{parity quantifier} $\Mod^1_2$. The syntax of
$\FOP$ augments the syntax of $\FO$ with the following formation rule: if $\varphi(y)$ is
a $\FOP$-formula, then $\oplus y \varphi(y)$ is also a $\FOP$-formula; this formula is true if the number
of $y$'s that satisfy $\varphi(y)$ is odd (analogously, $\Mod_q^i y \varphi(y)$ is true if the number of $y$'s that satisfy $\varphi(y)$ is congruent to $i \mod q$). A typical property on graphs expressible in $\FOP$ (but not in $\FO$) is
$\mathcal P:= \{G : \mbox{ every vertex of $G$ has odd degree}\}$,  since a graph is in $\mathcal P$ if and only if it satisfies the $\FOP$-sentence $\forall x \oplus y E(x,y)$.

Our main result (see Theorem~\ref{thm:Main}) is a \emph{modular convergence law} for $\FOP$ on $G(n,p)$ with $p$ a constant in $(0,1)$.
This law asserts that if $\varphi$ is a $\FOP$-sentence, then there are two explicitly computable rational numbers $a_0$, $a_1$, such that, as $n\rightarrow \infty$, the probability that the random graph $G(2n+i, p)$ satisfies $\varphi$ approaches $a_i$, for $i=0,1$. Moreover, $a_0$ and $a_1$ are computable and are of the form $r/2^{s}$, where $r$ and $s$
are non-negative integers.
We also establish that an analogous modular convergence law
holds for every extension $\FOM$ of $\FO$ with the counting quantifiers $\{\Mod^i_q : i \in [q-1]\}$, where
$q$ is a prime. It should be noted that results in \cite{HKL96} imply that the modular convergence law for $\FOP$ does \emph{not}
generalize to extensions of $\FOP$ with fixed-point operators. This is in sharp contrast to the aforementioned
0-1 law for $\FO$ which carries over to extensions of $\FO$ with fixed-point operators.

\subsection{Methods}

Earlier 0-1 laws have been established by a combination of standard methods and techniques
from mathematical logic and random graph theory. In particular, on the side of mathematical
logic, the tools used include the compactness theorem, Ehrenfeucht-Fra\"{i}ss\'e games, and
quantifier elimination. Here, we establish the modular convergence law by combining
quantifier elimination with, interestingly, algebraic methods related to
multivariate polynomials over finite fields. In what follows in this section, we
present an overview of the methods and techniques that we will use.

\subsubsection{The distribution of subgraph frequencies mod $q$, polynomials and Gowers norms}
Let us briefly indicate the relevance of polynomials to the study of $\FOP$
on random graphs. A natural example of a statement in $\FOP$ is a formula $\varphi$
such that $G$ satisfies $\varphi$ if and only if the number of copies of $H$ in
$G$ is odd, for some graph $H$ (where by copy we mean an induced subgraph, for now).
Thus understanding the asymptotic probability of $\varphi$ on $G(n, p)$ amounts
to understanding the distribution of the number of copies (mod $2$) of $H$ in $G(n,p)$.

In this spirit, we ask: what is the probability
that in $G(n, 1/2)$ there is an odd number of
triangles (where we count {\em unordered} triplets of vertices $\{a, b, c\}$ such that $a, b, c$ are all pairwise adjacent\footnote{Counting the number of {\em unordered} triples is not expressible in $\FOP$, we ask
this question only for expository purposes (nevertheless, we do give an answer to this question
in Section~\ref{sec:equi}).})?

We reformulate this question in terms of the following
``triangle polynomial", that takes the adjacency matrix of a graph as input
and returns the parity of the number of triangles in the graph;
$P_{\triangle} : \{0,1\}^{n \choose 2} \rightarrow \{0,1\}$, where
$$P_{\triangle}((x_{e})_{e \in {n \choose 2}}) = \sum_{\{e_1, e_2, e_3\}\mbox{ forming a }\triangle} x_{e_1} x_{e_2} x_{e_3},$$
where the arithmetic is$\mod 2$. Note that for the random graph $G(n,1/2)$, each entry of the adjacency matrix is
chosen independently and uniformly from $\{0,1\}$. Thus the probability that a random graph $G \in G(n,1/2)$ has an odd number of triangles
is precisely equal to
$\Pr_{x \in \Z_2^n}[ P_{\triangle}(x) = 1 ]$. Thus we have reduced our problem to studying
the distribution of the evaluation of a certain polynomial at a random point, a topic of
much study in pseudorandomness and algebraic coding theory, and we may now appeal to tools from
these areas. 

In Section~\ref{sec:equi}, via the above approach, we show that the probability
that $G(n,1/2)$ has an odd number of triangles equals $1/2 \pm 2^{-\Omega(n)}$.
Similarly, for any connected graph $F \neq \K_1$ (the graph consisting of one
vertex), the probability that $G(n,1/2)$ has an odd number of copies\footnote{
with a certain precise definition of ``copy".} of $F$ is
also $1/2 \pm 2^{-\Omega(n)}$ (when $F = \K_1$, there is no randomness in
the number of copies of $F$ in $G(n,1/2)$!). In fact,
we show that for any collection of distinct connected graphs $F_1, \ldots, F_\ell$
($\neq \K_1$), the joint distribution of the number of copies mod $2$ of $F_1, \ldots, F_\ell$
in $G(n,1/2)$ is $2^{-\Omega(n)}$-close to the uniform distribution on $\Z_2^\ell$, i.e.,
the events that there are an odd number of $F_i$ are essentially independent of one another.

Generalizing the above to $G(n,p)$ and counting mod $q$ for arbitrary
$p \in (0,1)$ and arbitrary integers $q$ motivates the study of
new kinds of questions about polynomials, that we believe are interesting
in their own right. For $G(n, p)$ with arbitrary $p$,
we need to study the distribution of $P(x)$, for certain polynomials $P$,
when $x \in \Z_2^m$ is distributed according to the {\em $p$-biased measure}.
Even more interestingly, for the study of $\FOM$,
where we are interested in the distribution of the number of triangles$\mod q$,
one needs to understand the distribution of $P(x)$ ($P$ is
now a polynomial over $\Z_q$) where $x$ is chosen uniformly from $\{0,1\}^{m}
\subseteq \Z_q^{m}$ (as opposed to $x$ being chosen uniformly from
all of $\Z_q^m$, which is traditionally studied). In 
Section~\ref{asec:poly}, we develop all the
relevant polynomial machinery
in order to answer these questions. This involves generalizing some
classical results of Babai, Nisan and Szegedy~\cite{BNS}
on correlations of polynomials. The key technical innovation here is
our definition of a $\mu$-Gowers norm (where $\mu$ is a measure on $\Z_q^m$)
that measures the correlation, under $\mu$, of a given function
with low-degree polynomials (letting $\mu$ be the uniform measure, we recover
the standard Gowers norm). After generalizing several results about
the standard Gowers norm to the $\mu$-Gowers norm case, we can then use a
technique of Viola and Wigderson~\cite{VW} to establish
the generalization of~\cite{BNS} that we need.

\subsubsection{Quantifier elimination}

Although we studied the distribution of subgraph frequencies
mod $q$ as an attempt to determine the limiting behavior of only
a special family of $\FOM$ properties, it turns out that
this case, along with the techniques developed to handle it,
play a central role in the proof of the full modular convergence law.
In fact, we reduce the modular convergence law for general $\FOM$ properties
to the above case.
We show that for any $\FOM$ sentence $\varphi$, with
high probability over $G \in G(n,p)$, the truth of $\varphi$
on $G$ is determined by the number of copies in $G$, mod $q$, of
each small subgraph. Then by the results described earlier
on the equidistribution of these numbers (except for the number of $\K_1$,
which depends only on $n \mod q$), the full modular
convergence law for $\FOM$ follows.

In Section~\ref{sec:quant}, we establish such a reduction using the method of
elimination of quantifiers. To execute this, we need to analyze $\FOM$
formulas which may contain free variables (i.e., not every variable
used is quantified).
Specifically, we show that for every
$\FOM$ formula $\varphi(\alpha_1, \ldots, \alpha_k)$,
with high probability over $G \in G(n,p)$, it holds that for
all vertices $w_1, \ldots, w_k$ of $G$, the truth of $\varphi(w_1, \ldots, w_k)$
is entirely determined by the following data: (a) which of the
$w_i, w_j$ pairs are adjacent, (b) which of the $w_i$, $w_j$ pairs
are equal to one another, and (c) the number of copies ``rooted'' at $w_1, \ldots, w_k$,
mod $q$, of each small {\em labelled graph}.
This statement is a generalization of what we needed to prove, but lends itself to
inductive proof ({\em this} is quantifier elimination).
This leads us to studying the distribution
(via the polynomial approach described earlier) of the number of
copies of {labelled graphs} in $G$; questions
of the form, given two specified vertices $v, w$ (the ``roots"),
what is the probability that there are an odd number of paths of length
$4$ in $G \in G(n,p)$ from $v$ to $w$?
After developing the necessary results on the distribution of
labelled subgraph frequencies, combined with some elementary combinatorics,
we can eliminate quantifiers and thus complete the proof of the modular
convergence law.


\subsection{Comparison with $\AC$}

Every $\FOP$ property naturally defines a family of boolean functions
$f_n : \{0,1\}^{n \choose 2} \rightarrow \{0,1\}$, such that a graph
$G$ satisfies $\varphi$ if and only if $f_n(A_G) = 1$, where $A_G$ is the
adjacency matrix of $G$. This family of functions is easily seen to
be contained in 
$\mathsf {AC0}[\oplus]$, which is $\mathsf {AC0}$ with parity gates 
(each $\forall$ becomes an $\mathsf{AND}$ gate, $\exists$ becomes a $\mathsf{OR}$
gate and $\oplus$ becomes a parity gate). This may be summarized by saying
that $\FOP$ is a highly uniform version of $\AC$.

Currently, all our understanding of the power of $\AC$ comes from the
Razborov-Smolensky \cite{razbo, smolen} approach to proving circuit lower
bounds on $\AC$. At the heart of this approach is the result that
for every $\AC$ function $f$, there is a low-degree polynomial $P$ such
that for $1 - \epsilon(n)$ fraction of inputs, the evaluations of $f$ and $P$ are equal.
Note that this result automatically holds for $\FOP$
(since $\FOP \subseteq \AC$).

We show that for the special case when $f: \{0,1\}^{n \choose 2} \rightarrow \{0,1\}$
comes from an $\FOP$ property $\varphi$, a significantly improved approximation may be obtained:
{\it(i)} We show that the degree of $P$ may be chosen to
be a constant depending only on $\varphi$, whereas the Razborov-Smolensky
approximation required $P$ to be of $\mathrm{polylog}(n)$ degree,
{\it(ii)} The error parameter $\epsilon(n)$ may be chosen
to be exponentially small in $n$, whereas the Razborov-Smolensky method
only yields $\epsilon(n) = 2^{-\log^{O(1)}n}$.
 {\it (iii):} Finally, the polynomial $P$ can be chosen to be
symmetric under the action of $S_n$ on the ${n \choose 2}$ coordinates,
while in general, the polynomial produced by the Razborov-Smolensky
approach need not be symmetric (due to the randomness involved in the
choices).

These strengthened approximation results allow us, using known results
about pseudorandomness against low-degree polynomials, to
show that {\it (i)} there exist explicit pseudorandom generators that
fool $\FOP$ sentences, and {\it (ii)} there exist explicit functions
$f$ such that for any $\FOP$ formula $\varphi$, the
probability over $G \in G(n,p)$ that $f(G) = \varphi(G)$ is at most
$\frac{1}{2} + 2^{-\Omega(n)}$.
The first result follows from the pseudorandom generators against
low-degree polynomials due to Bogdanov-Viola~\cite{BogVio}, Lovett~\cite{Lov}
and Viola~\cite{Vio}. The second result follows from the result of
Babai, Nisan and Szegedy~\cite{BNS}, and our generalization of it,
giving explicit functions that are uncorrelated with low degree polynomials.

Obtaining similar results for $\AC$ is one of the primary goals of
modern day ``low-level" complexity theory.

\paragraph{Organization of this paper:} In the next section, we
formally state our main results and some of its corollaries.
In Section~\ref{sec:equi}, we determine the distribution of subgraph
frequencies mod $q$. In Section~\ref{asec:poly}, we introduce the
$\mu$-Gowers Norm use it to prove some technical results on the bias of
polynomials needed for the previous section.
In Section~\ref{sec:quantstatement}, we state the theorem which
implements the quantifier elimination and describe the plan for its proof.
This plan is then executed in Sections~\ref{sec:quant},
\ref{sec:extendcount} and \ref{sec:indep}. We conclude with some
open questions.

\section{The Modular Convergence Law}

We now state our main theorem.

\begin{theorem}\label{thm:Main}
Let $q$ be a prime.
Then for every $\FO[\Mod_q]$-sentence $\varphi$, there exist
rationals $a_0, \ldots, a_{q-1}$ such that for every
$p \in (0,1)$ and every $i \in \{0,1, \ldots, q-1\}$,
$$ \lim_{\substack{n \rightarrow \infty\\ n\equiv i\  \mathrm{mod}\  q}} \Pr_{G \in G(n, p)}[ G \mbox{ satisfies } \varphi] = a_i.$$
\end{theorem}

\begin{remark}
The proof of Theorem~\ref{thm:Main} also yields:
\begin{itemize}
\item Given the formula $\varphi$, the numbers $a_0, \ldots, a_{q-1}$ can be computed.
\item Each $a_i$ is of the form $r/q^s$, where $r, s$ are nonnegative integers.
\item For every sequence of numbers $b_0, \ldots, b_{q-1} \in [0,1]$, each of the form $r/q^s$, there is a
$\FO[\Mod_q]$-sentence $\varphi$ such that for each $i$, the number $a_i$ given by the theorem
equals $b_i$.
\end{itemize}
\end{remark}

Before we describe the main steps in the proof of Theorem~\ref{thm:Main}, we make
a few definitions.

For graphs $F = (V_F, E_F)$ and $G = (V_G, E_G)$, an {\em (injective) homomorphism} from $F$ to $G$ is an (injective) map $\chi : V_F \rightarrow V_G$ that maps edges to edges, i.e., for any $(u,v) \in E_F$, we have $(\chi(u), \chi(v)) \in E_G$. Note that we do not require that $\chi$ maps non-edges to non-edges.
We denote by $[F](G)$ the number of injective homomorphisms from $F$ to $G$, and we denote by $[F]_q(G)$ this number$\mod q$. We let $\aut(F) := [F](F)$ be the number of automorphisms of $F$.


%

The following lemma (which follows from Lemma~\ref{lem:aut} in Section~\ref{sec:quant}), shows that for some graphs $F$, as $G$ varies, the number $[F](G)$ cannot be arbitrary.
\begin{lemma}\label{lem:aut-ident}
Let $F$ be a connected graph and $G$ be any graph.
Then $\aut(F)  \mid [F](G)$.
\end{lemma}
For the rest of this section, let $q$ be a fixed prime. 
Let $\Conn^a$ be the set of connected graphs on at most $a$ vertices. 
For any graph $G$, 
let the {\em subgraph frequency vector} $\freq_G^a \in \Z_q^{\Conn^a}$ be the vector
such that its value in coordinate $F$ ($F \in \Conn^a$) equals $[F]_q(G)$,
the number of injective homomorphisms from $F$ to $G$ mod $q$. Let $\Freq(a)$, the set of {\em feasible frequency vectors},
be the subset of $\Z_q^{\Conn^a}$ consisting of all $f$ such that
for all $F \in \Conn^a$, $ f_{F} \in \aut(F) \cdot \Z_q := \{ \aut(F) \cdot x \mid x \in \Z_q\}$. 
By Lemma~\ref{lem:aut-ident}, for every $G$ and $a$, $\freq_G^a \in \Freq(a)$, i.e., the
subgraph frequency vector is always feasible.


We can now state the two main technical results
that underlie Theorem~\ref{thm:Main}.

The first states that on almost all graphs $G$, every $\FO[\Mod_q]$ formula can be expressed in terms of the subgraph frequencies, $[F]_q(G)$, over all small connected graphs $F$.

\begin{theorem}{\bf (Subgraph frequencies$\mod q$ determine $\FO[\Mod_q]$ formulae)}
\label{thm:Main1}
For every $\FO[\Mod_q]$-sentence $\varphi$ of quantifier depth $t$, there exists an integer $c = c(t,q)$ and a function
$\psi : \Z_q^{\Conn^c} \rightarrow \{0, 1\}$ such that for all $p \in (0,1)$,
$$\Pr_{G \in G(n,p)}\left[ (G\mbox{ satisfies } \varphi) \Leftrightarrow (\psi(\freq^c_G) = 1) \right] \geq 1 - \exp(-n).$$
\end{theorem}

This result is complemented by the following result, that
shows the distribution of subgraph frequencies$\mod q$ in a random graph $G \in G(n,p)$ is essentially uniform in the space of all feasible frequency vectors, up to the obvious restriction that the number of vertices (namely the frequency of $\K_1$ in $G$) should equal $n \mod q$.
\begin{theorem}[Distribution of subgraph frequencies$\mod q$ depends only on $n \mod q$]
\label{thm:Main2}
Let $p \in (0,1)$. Let $G \in G(n, p)$. Then for any constant $a$, the distribution of
$\freq_G^a$ is $\exp(-n)$-close to the uniform distribution over the set
$$\{ f \in \Freq(a) : f_{K_1} \equiv n \mod q \}.$$
\end{theorem}

Theorem~\ref{thm:Main2} is proved in Section~\ref{sec:equi} by studying the bias of
 multivariate polynomials over finite fields via a generalization of the Gowers norm.
Theorem~\ref{thm:Main1} is proved in Section~\ref{sec:quant} using two main ingredients:
\begin{enumerate}
\item A generalization of Theorem~\ref{thm:Main2} that determines the joint distribution
of the frequencies of ``labelled subgraphs" with given roots (see Section~\ref{sec:indep}).
\item A variant of quantifier elimination (that may be called quantifier conversion) designed to handle $\Mod_q$ quantifiers that crucially uses the probabilistic input from the previous ingredient
(see Section~\ref{sec:quant}).
\end{enumerate}

\begin{proofof}{Theorem~\ref{thm:Main}}
Follows by combining Theorem~\ref{thm:Main1} and Theorem~\ref{thm:Main2}.
\end{proofof}

\subsection{Pseudorandomness against $\FOP$}

We now point out three simple corollaries of our study of $\FOP$ on random graphs.

\begin{corollary}[$\FOM$ is well approximated by low-degree polynomials]
\label{lem:lowdegapprox}
For every $\FOM$-sentence $\varphi$, there is a constant $d$, such that for each $n \in \mathbb N$, there is a degree $d$ polynomial $P( (X_e)_{e \in {n \choose 2}}) \in \Z_q[(X_e)_{e \in {n\choose 2}}]$, such that for all $p \in (0,1)$,
$$ \Pr_{G \in G(n,p)} [ (G \mbox{ satisfies } \varphi) \Leftrightarrow P(A_G)=1] \geq 1 - 2^{-\Omega(n)},$$
where $A_G \in \{0,1\}^{n \choose 2}$ is the adjacency matrix of $G$.
\end{corollary}
\begin{proof}
Follows from Theorem~\ref{thm:Main1} and the observation that for any graph $F$ of constant size, there is a polynomial $Q((X_e)_{e \in {n \choose 2}})$ of constant degree, such that $Q(A_G) = [F]_q(G)$ for all graphs $G$.
\end{proof}

\begin{corollary}[PRGs against $\FOP$]
For each $s \in \mathbb N$ and constant $\epsilon > 0$, there is a constant $c \geq 0$ such that for each $n$, there is a family $\mathcal F$ of $\Theta(n^c)$ graphs on $n$ vertices, computable in time $\mathrm{poly}(n^c)$, such that for all $\FOP$-sentences $\varphi$ of size at most $s$, and for all $p \in (0,1)$,
$$|\Pr_{G \in \mathcal F}[G\mbox{ satisfies }\varphi] - \Pr_{G \in G(n,p)}[G \mbox{ satisfies }\varphi]| < \epsilon.$$
\end{corollary}
\begin{proof}
For $p = 1/2$, this follows from the previous corollary and the result of Viola~\cite{Vio} (building on results of Bogdanov-Viola~\cite{BogVio} and Lovett~\cite{Lov}) constructing a pseudorandom generator fooling low-degree polynomials
under the uniform distribution. For 
general $p$, note that the same family $\mathcal F$ from the $p = 1/2$ case works, since the distribution
of subgraph frequencies given in Theorem~\ref{thm:Main2} is independent of $p$.
\end{proof}

The analogue of the previous corollary for $\FO$ was proved in~\cite{GS-paley, BEH-paley} (see also~\cite{BR-extension, NNT}).

\begin{corollary}[Explicit functions exponentially hard for $\FOP$]
There is an explicit function $f : \{0,1\}^{n \choose 2} \rightarrow \{0,1\}$
such that for every $\FOP$-sentence $\varphi$,
$$\Pr_{G \in G(n,p)}[ (G\mbox{ satisfies }\varphi) \Leftrightarrow (f(A_G) = 1)] < \frac12 + 2^{-\Omega(n)}.$$
\end{corollary}
\begin{proof}
Follows from Corollary~\ref{lem:lowdegapprox}, and the result of Babai, Nisan, Szegedy~\cite{BNS}
(for $p = 1/2$) and its generalization, Lemma~\ref{lem:VW} (for general $p$),
constructing functions exponentially uncorrelated with low degree polynomials under the $p$-biased measure.
It actually follows from our proofs that, one may even choose a function $f$ that is a graph property (namely, invariant under the action of $S_n$ on the coordinates).
\end{proof}

\section{The Distribution of Subgraph Frequencies mod q}
\label{sec:equi}

In this section, we prove Theorem~\ref{thm:Main2}
on the distribution of subgraph frequencies in $G(n,p)$.

We first make a few definitions. If $F$ is a connected graph and
$G$ is any graph,  a {\em copy} of $F$ in $G$ is a
set $E \subseteq E_G$ such that there exists an injective homomorphism $\chi$ from $F$ to $G$ such that $E = \chi(E_F) := \{(\chi(v), \chi(w)) \mid (v, w) \in E_F\}$. We denote the set of copies of $F$ in $G$ by $\Cop(F,G)$,
the cardinality of $\Cop(F,G)$ by $\langle F \rangle (G)$,
and this number mod $q$ by $\langle F \rangle_q(G)$. We have the following basic relation (which follows from Lemma~\ref{lem:aut} in Section~\ref{sec:quant}).
\begin{lemma}
\label{lem:autunlab}
If $F$ is a connected graph with $|E_F| \geq 1$, then
$$ [F](G) = \aut(F) \cdot \langle F \rangle(G).$$
\end{lemma}

For notational convenience, we view $G(n, p)$ as a graph
whose vertex set is $[n]$ and whose edge set is
a subset of ${[n] \choose 2}$.

We can now state the general equidistribution theorem from which Theorem~\ref{thm:Main2} will follow easily (We use the notation $\Omega_{q,p,d}(n)$ to denote
the expression $\Omega(n)$, where the implied constant depends only on $q$, $p$ and $d$). Note that this theorem holds for arbitrary integers $q$, not necessarily prime.
\begin{theorem}[Equidistribution of graph copies]
\label{thm:uniform}
Let $q > 1$ be an integer and let $p \in (0,1)$.
Let $F_1, \ldots, F_\ell \in \Conn^a$ be distinct graphs with $1 \leq |E_{F_i}| \leq d$.

Let $G \in G(n,p)$.
Then the distribution of $(\langle F_1\rangle_q(G), \ldots, \langle F_\ell \rangle_q(G))$ on $\Z_q^\ell$ is
$2^{-\Omega_{q, p, d}(n) + \ell}$-close to uniform in statistical distance.
\end{theorem}

Using this theorem, we complete the proof of Theorem~\ref{thm:Main2}.

\begin{proofof}{Theorem~\ref{thm:Main2}}
Let $F_1, \ldots, F_\ell$ be an enumeration of the elements of $\Conn^a$ except for $\K_1$.
By Theorem~\ref{thm:uniform}, the distribution of $g = (\langle F_i\rangle_q(G)_{i=1}^{\ell}$ is $2^{-\Omega(n)}$ close to uniform over $\Z_q^\ell$. Given the
vector $g$, we may compute the vector $\freq_G^{a}$ by:
\begin{itemize}
\item $(\freq_{G}^a)_{\K_1} = n \mod q$.
\item For $F \in \Conn^a \setminus\{\K_1\}$, $(\freq_G^a)_{F} = g_{F} \cdot \aut(F)$ (by Lemma~\ref{lem:autunlab}).
\end{itemize}
This implies that the distribution of $\freq_G^a$ is $2^{-\Omega(n)}$-close to uniformly distributed over
$\{ f \in \Freq(a): f_{\K_1} = n \mod q\}$.
\end{proofof}

Towards proving Theorem~\ref{thm:uniform}, we now introduce some tools.
\subsection{Preliminary lemmas}

As indicated in the introduction, the distribution
of subgraph frequencies is most naturally studied
via the distribution of values of certain polynomials.
The following lemma, which is used in the proof of Theorem~\ref{thm:uniform}
(and again in Section~\ref{sec:indep} to study the distribution of labelled subgraph frequencies),
gives a simple sufficient criterion
for the distribution of values of a polynomial to be ``unbiased".
The proof appears in Section~\ref{asec:poly}.
\begin{lemma}
\label{lem:polybias1}
Let $q > 1$ be an integer and let $p \in (0,1)$.
Let\footnote{If $S$ is a set, we use the notation $2^{S}$ to denote its power set.} $\mathcal F \subseteq 2^{[m]}$. Let $d > 0$ be an integer.
Let $Q(Z_1, \ldots, Z_m) \in \Z_q[Z_1, \ldots, Z_m]$ be a polynomial of the form
$$\sum_{S \in \mathcal F} a_S \prod_{i \in S} Z_i + Q'(\mathbf Z),$$
where $\deg(Q') < d$.
Suppose there exist $\mathcal E = \{E_1, \ldots, E_r\} \subseteq \mathcal F$ such that:
\begin{itemize}
\item $|E_j| = d$ for each $j$,
\item $a_{E_j} \neq 0$ for each $j$.
\item $E_j \cap E_{j'} = \emptyset$ for each $j, j'$,
\item For each $S \in \mathcal F\setminus\mathcal E$, $|S \cap (\cup_j E_j)| < d$.
\end{itemize}
Let $\mathbf z = (z_1, \ldots, z_m) \in \Z_q^m$ be the random variable
where, independently for each $i$, we have $\Pr[z_i = 1] = p$ and $\Pr[z_i = 0] = 1-p$.
Then,
$$\left|\mathbb E \left[ \omega^{Q(\mathbf z)} \right]\right| \leq 2^{-\Omega_{q,p,d}(r)},$$
where $\omega \in \mathbb C$ is a primitive $q^{\rm{th}}$-root of unity.
\end{lemma}

The lemma below is a useful tool for showing that a distribution
on $\Z_q^\ell$ is close to uniform.
\begin{lemma}[Vazirani XOR lemma]
\label{lem:xor}
Let $q > 1$ be an integer and let $\omega \in \mathbb C$ be a primitive
$q^{\rm{th}}$-root of unity. Let $\mathbf X = (X_1, \ldots, X_\ell)$ be a random variable over $\Z_q^\ell$.
Suppose that for every nonzero $c \in \Z_q^\ell$,
$$\left|\E \left[\omega^{\modsum_{i \in [\ell]} c_i X_i}\right]\right| \leq \epsilon.$$
Then $\mathbf X$ is $q^\ell \cdot \epsilon$-close to uniformly
distributed over $\Z_q^\ell$.
\end{lemma}

\subsection{Proof of the equidistribution theorem}

\begin{proofof}{Theorem~\ref{thm:uniform}}
By the Vazirani XOR Lemma (Lemma~\ref{lem:xor}), it suffices to show
that for each nonzero $c \in \Z_q^{\ell}$, we have
$\left|\E\left[ \omega^R\right] \right| \leq 2^{-\Omega_{q,p,d}(n)}$, 
where $R := \modsum_{i\in[\ell]} c_i \langle F_i\rangle_q(G)$,
and $\omega \in \mathbb C$ is a primitive $q^{\rm{th}}$-root of unity.

We will show this by appealing to Lemma~\ref{lem:polybias1}. Let
$m = {n \choose 2}$.
Let $\mathbf z \in \{0,1\}^{[n] \choose 2}$ be the random variable
where, for each $e \in { [n] \choose 2}$, $z_e = 1$ if and only
if $e$ is present in $G$. Thus, independently for each $e$,
$\Pr[z_e = 1] = p$.

We may now express $R$ in terms of the $z_e$. Let $\K_n$
denote the complete graph on the vertex set $[n]$.
Thus $\Cop(F_i, \K_n)$
is the set of $E$ that could potentially arise as copies of
$F_i$ in $G$. Then we may write,
\begin{align*}
R = \modsum_{i \in [\ell]} c_i \langle F_i \rangle_q(G) &= \modsum_{i \in [\ell]} c_i \sum_{E \in \Cop(F_i, \K_n)} \prod_{e \in E} z_e\\
&= \modsum_{E \in \mathcal F} c_E \prod_{e \in E} z_e,
\end{align*}
where $\mathcal F \subseteq 2^{[n] \choose 2}$ is the set $\bigcup_{i: c_i \neq 0} \Cop(F_i, \K_n)$,
and for $E \in \mathcal F$, $c_E = c_i$ for the unique $i$ satisfying $E \in \Cop(F_i, \K_n)$
(note that since the $F_i$ are nonisomorphic connected graphs, the $\Cop(F_i, \K_n)$ are pairwise
disjoint).

Let $Q(\mathbf Z) \in \Z_q[\mathbf Z]$, where $\mathbf Z = (Z_e)_{e \in {[n]\choose 2}}$  be the polynomial
$\sum_{E \in \mathcal F} c_E \prod_{e \in E} Z_e$.
Then $R = Q(\mathbf z)$. We wish to show that
\begin{equation}
\label{eq:uniform}
\left| \E\left[ \omega^{Q(\mathbf z)}\right] \right| \leq 2^{-\Omega_{q,p,d}(n)}.
\end{equation}

We do this by demonstrating that the polynomial $Q(\mathbf Z)$ satisfies the hypotheses of Lemma~\ref{lem:polybias1}.

Let $d^* = \max_{i: c_i \neq 0} |E_{F_i}|$.
Let $i_0 \in [\ell]$ be such that $c_{i_0} \neq 0$ and $|E_{F_{i_0}}| = d^{*}$. Let
$\chi_1, \chi_2, \ldots, \chi_r \in \Inj(F_{i_0}, \K_n)$ be a collection of homomorphisms such that for
all distinct $j, j' \in [r]$, we have $\chi_j(V_{F_{i_0}}) \cap \chi_{j'}(V_{F_{i_0}}) = \emptyset$. Such a collection can be chosen greedily so that $r = \Omega(\frac{n}{d})$.
Let $E_j \in \Cop(F_{i_0},\K_n)$ be given by $\chi_j(E_{F_{i_0}})$.
Let $\mathcal E$ be the family of sets $\{E_1, \ldots, E_r\} \subseteq \mathcal F$. We observe the following properties of the $E_j$:
\begin{enumerate}
\item For each $j \in [r]$, $|E_j| = d^*$ (since $\chi_j$ is injective).
\item For each $j \in [r]$, $c_{E_j} = c_{i_0} \neq 0$.
\item For distinct $j, j' \in [r]$, $E_j \cap E_{j'} = \emptyset$ (by choice of the $\chi_j$).
\item For every $S \in \mathcal F \setminus \mathcal E$, $|S \cap (\cup_j E_j)| < d^*$. To see this, take any $S \in \mathcal F \setminus \mathcal E$ and suppose $|S \cap (\cup_j E_j)| \geq d^*$.
Let $i' \in [\ell]$ be such that $c_{i'} \neq 0$ and $S \in \Cop(F_{i'}, \K_n)$. Let $\chi \in \Inj(F_{i'}, \K_n)$
with $\chi(E_{F_{i'}}) = S$. By choice of $d^*$,
we know that $|S| \leq d^*$.
Therefore, the only way that $|S \cap (\cup_j E_j)|$ can be $\geq d^*$ is
if (1) $|S| = d^*$, and (2) $S \cap (\cup_j E_j) = S$, or in other words, $S \subseteq (\cup_j E_j)$.
However, since the $\chi_j(V_{F_{i_0}})$ are all pairwise disjoint, this implies that $S \subseteq E_j$
for some $j$. But since $|E_j| = |S|$, we have $S = E_j$, contradicting our choice of $S$.
Therefore, $|S \cap (\cup_j E_j)| < d^*$ for any $S \in \mathcal F \setminus \mathcal E$.
\end{enumerate}

It now follows that
 $Q(\mathbf Z)$, $\mathcal F$ and $\mathcal E$ satisfy the hypothesis of
Lemma~\ref{lem:polybias1}. Consequently, (recalling that $r = \Omega(n/d)$ and $d^*\leq d$) Equation~\eqref{eq:uniform} follows, completing
the proof of the theorem.
\end{proofof}

\begin{remark}
We just determined the joint distribution of the number of injective
homomorphisms, mod $q$, from all small connected graphs to $G(n,p)$. 
This information can be used in conjunction with 
Lemma~\ref{lem:kgraphtoconnected} to determine the joint distribution
of the number of injective homomorphisms, mod $q$, from {\em all} small graphs
to $G(n,p)$.
\end{remark}
\section{The Bias of Polynomials}
\label{asec:poly}

Our main goal in this section is to give a full proof of Lemma~\ref{lem:polybias1},
which gives a criterion for a polynomial to be unbiased. Along the way,
we will introduce the $\mu$-Gowers norm and   some of its useful properties.

Our proof of Lemma~\ref{lem:polybias1} will go through
the following lemma (which is proved in the next 
subsection). It shows that ``Generalized Inner Product" polynomials
are uncorrelated with polynomials of lower degree.
This generalizes a result of Babai Nisan and Szegedy~\cite{BNS} (which
dealt with the case $q = 2$ and $p = 1/2$).
\begin{lemma}
\label{lem:VW}
Let $q> 1$ be an integer and let $p \in (0,1)$.
Let $E_1, \ldots, E_r$ be pairwise disjoint subsets of $[m]$ each
of cardinality $d$.
Let $Q(Z_1, \ldots, Z_m) \in \Z_q[Z_1, \ldots, Z_m]$ be a polynomial of the form
$$\left(\sum_{j=1}^r a_j \prod_{i \in E_j} Z_i\right) + R(\mathbf Z),$$
where each $a_j \neq 0$ and $\deg(R(\mathbf Z)) < d$.
Let $\mathbf z = (z_1, \ldots, z_m) \in \Z_q^m$ be the random variable
where, independently for each $i$, we have $\Pr[z_i = 1] = p$ and $\Pr[z_i = 0] = 1-p$.
Then,
$$\left|\mathbb E \left[ \omega^{Q(\mathbf z)} \right]\right| \leq 2^{-\Omega_{q,p,d}(r)}.$$
\end{lemma}


Given Lemma~\ref{lem:VW}, we may now prove Lemma~\ref{lem:polybias1}.

\begin{proofof}{Lemma~\ref{lem:polybias1}}
Let $U = \cup_{j=1}^{r}  E_j$.
Fix any $x \in \{0,1\}^{[m]\setminus U}$, and
let $Q_{x}(\mathbf Y) \in \Z_q[(Y_i)_{i \in U}]$ be the polynomial
$$ \sum_{S \in \mathcal F}  a_S \left(\prod_{j \in S \cap ([m]\setminus U)}x_j\right)\left(\prod_{i \in S \cap U} Y_i\right) + Q'(x, \mathbf Y)$$
so that $Q_x(y) = Q(x,y)$ for each $y \in \Z_q^U$.
Notice that the degree (in $\mathbf Y$) of the term corresponding to $S \in \mathcal F$ is $|S \cap U|$. By assumption, unless $S = E_j$ for some $j$,
we must have $|S \cap U| < d$.

Therefore the polynomial $Q_{x}(\mathbf Y)$ is of the form:
$$\sum_{j = 1}^r a_{E_i} \prod_{i \in E_i} Y_i + R(\mathbf Y),$$
where $\deg(R(\mathbf Y)) < d$. By Lemma~\ref{lem:VW},
$$\left|\mathbb E\left[\omega^{Q_x(\mathbf y)}\right]\right| < 2^{-\Omega_{q,p,d}(r)},$$
where $\mathbf y \in \{0,1\}^U$ with each $y_i = 1$ independently with
probability $p$.

As $Q_x(y) = Q(x,y)$, we get
$$\left|\mathbb E\left[\omega^{Q(\mathbf z^x)}\right]\right| < 2^{-\Omega_{q,p,d}(r)},$$
where $\mathbf z^x \in \Z_q^{n}$ is the random variable
$\mathbf z$ conditioned on the event $z_{j}= x_j$ for every
$j \in [m]\setminus U$.
Now, the distribution of $\mathbf z$ is a convex combination of the distributions of
$\mathbf z^x$ as $x$ varies over $\{0,1\}^{[m]\setminus U}$.
This allows us to deduce that
$$\left|\mathbb E \left[ \omega^{Q(\mathbf z)} \right]\right| \leq 2^{-\Omega_{q,p,d}(r)},$$
as desired.
\end{proofof}

\subsection{The $\mu$-Gowers norm}
\label{asec:gowers}
The proof of Lemma~\ref{lem:VW} will use a variant of
the Gowers norms.
Let $Q: \Z_q^m \rightarrow \Z_q$ be any function, and define
$f : \Z_q^m \rightarrow \mathbb C$ by $f(x) = \omega^{Q(x)}$.
The Gowers norm of $f$ is an analytic quantity that measures
how well $Q$ correlates with degree $d$ polynomials: 
the correlation of $Q$ with polynomials of degree $d-1$
under the uniform distribution is bounded from above by the
$d^{\rm{th}}$-Gowers norm of $f$. Thus to show that a certain $Q$ is uncorrelated
with all degree $d-1$ polynomials under the uniform distribution,
it suffices to bound the $d^{\rm{th}}$-Gowers norm of $f$.
In Lemma~\ref{lem:VW}, we wish to show that a certain $Q$ is uncorrelated
with all degree $d-1$ polynomials under a distribution $\mu$ that need not be
uniform. To this end, we define a variant of the Gowers norm,
which we call the $\mu$-Gowers norm, and show that if the $(d,\mu)^{\rm{th}}$-Gowers
norm of $f$ is small, then $Q$ is uncorrelated with all degree $d-1$
polynomials {\em under $\mu$}. We then complete the proof
of Lemma~\ref{lem:VW} by bounding the $(d,\mu)^{\rm{th}}$-Gowers
norm of the relevant $f$.

We first define the $\mu$-Gowers norm and develop some of its basic properties.

Let $H$ be an abelian group and let $\mu$ be a probability distribution on $H$.
For each $d \geq 0$, define a probability distribution $\mu^{(d)}$ on $H^{d+1}$ inductively by $\mu^{(0)} = \mu$, and, for $d \geq 1$, let $\mu^{(d)}(x, t_1, \ldots, t_d)$ equal
$$\frac{\mu^{(d-1)}(x, t_1, \ldots, t_{d-1}) \cdot \mu^{(d-1)}(x + t_d, t_1, \ldots, t_{d-1})}{\sum_{z \in H} \mu^{(d-1)}(z, t_1, \ldots, t_{d-1})}.$$
Equivalently, to sample $(x, t_1, \ldots, t_d)$ from $\mu^{(d)}$, first take a sample $(x, t_1, \ldots, t_{d-1})$
from $\mu^{(d-1)}$, then take a sample $(y, t'_1, \ldots, t'_{d-1})$ from $\mu^{(d-1)}$ conditioned on
$t'_i = t_i$ for each $i \in [d-1]$, and finally set $t_d = y-x$ (our sample
is then $(x, t_1, \ldots, t_{d-1}, t_d)$).
Notice that the distribution of a sample $(x, t_1, \ldots, t_d)$ from $\mu^{(d)}$
is such that for each $S \subseteq [d]$, the distribution of the point $x+ \sum_{i \in S} t_i$
is precisely $\mu$.

For a function $f: H \rightarrow \mathbb C$ and $\mathbf t \in H^d$, we define its
{\em $d^{\rm{th}}$-derivative in directions $\mathbf t$} to be the function
$D_{\mathbf t}f: H \rightarrow \mathbb C$ given by
$$D_{\mathbf t}f(x) = \prod_{S \subseteq [d]} f(x + \sum_{i \in S} t_i)^{\circ S},$$
where $a^{\circ S}$ equals the complex conjugate $\bar a$ if $|S|$ is odd, and $a^{\circ S}$ equals $a$ otherwise.
From the definition it immediately follows that $D_{(\mathbf t, u)}f (x) = D_{\mathbf t}f(x) \overline{D_{\mathbf t}f(x+u)}$
(where $(\mathbf t, u)$ denotes the vector $(t_1, \ldots, t_d, u) \in H^{d+1}$).

We now define the $\mu$-Gowers norm.
\begin{definition}[$\mu$-Gowers Norm]
If $\mu$ is a distribution on $H$, and $f : H \rightarrow \mathbb C$,
we define its $(d, \mu)$-Gowers norm by
$$\| f\|_{U^d, \mu} = \left|\mathbb E_{(x, \vec t) \sim \mu^{(d)}} \left[(D_{\vec t}f)(x)\right]\right|^{\frac{1}{2^d}}.$$
\end{definition}
When $H$ is of the form $\Z_q^m$, then the $(d,\mu)$-Gowers norm of a function is supposed
to estimate the correlation, under $\mu$, of that function with polynomials of degree $d-1$.
Intuitively, this happens because the Gowers norm of $f$ measures how often the $d^{\rm{th}}$
derivative of $f$ vanishes.

The next few lemmas enumerate some of the useful properties
that $\mu$-Gowers norms enjoy.

\begin{lemma}
\label{gowerslemma4}
Let $f : H \rightarrow \mathbb C$. Then,
$$\left| \mathbb E_{x \sim \mu} \left[ f(x)\right] \right| \leq \|f \|_{U^d, \mu}.$$
\end{lemma}
\begin{proof}
We prove that for every $d$,
$\| f \|_{U^d, \mu} \leq \| f \|_{U^{d+1}, \mu}$.
The lemma follows by noting that $\| f \|_{U^0, \mu} = \left| \mathbb E_{x \sim \mu} \left[ f(x) \right] \right|$.

The proof proceeds (following Gowers~\cite{gowers} and Green-Tao~\cite{greentao}) via the Cauchy-Schwarz inequality,
\begin{align*}
\| f \|_{U^d, \mu}^{2^{d+1}} &= \left|\mathbb E_{(x, \mathbf t) \sim \mu^{(d)}}\left[ D_{\mathbf t}f (x) \right] \right|^2\\
&\leq \mathbb E_{\mathbf t}\left[\left|\mathbb E_{x}\left[D_{\mathbf t}f(x) \right]   \right|^2 \right]\mbox{\quad\quad\quad\quad\quad\quad\quad by Cauchy-Schwarz}\\
&= \mathbb E_{\mathbf t} \mathbb E_{x, y} \left[D_{\mathbf t}f(x) \overline{D_{\mathbf t}f(y)}  \right] \mbox{\quad\quad \quad\quad\quad where $y$ is an independent sample of $x$ given $\mathbf t$ }.\\
&= \mathbb E_{x, \mathbf t, u}\left[ D_{\mathbf t}f(x) \overline{D_{\mathbf t}f(x + u)}\right]\mbox{\quad\quad\quad\quad where $u = y - x$}\\
&= \mathbb E_{(x, \mathbf t, u) \sim \mu^{(d+1)}}\left[ D_{\mathbf t}f(x) \overline{D_{\mathbf t}f(x + u)}\right]\mbox{\quad\quad by definition of $\mu^{(d+1)}$}\\
&= \mathbb E_{(x, \mathbf t, u) \sim \mu^{(d+1)}}\left[ D_{(\mathbf t, u)}f(x) \right]\\
&= \|f\|_{U^{d+1}, \mu}^{2^{d+1}}.
\end{align*}
This proves the lemma.
\end{proof}

\begin{definition}
For each $i \in [r]$, let $g_i:H \rightarrow \mathbb C$.
We define $(\bigotimes_{i=1}^r g_i): H^r \rightarrow \mathbb C$ by
$$\left(\bigotimes_{i=1}^r g_i\right)(x_1, \ldots, x_r) = \prod_{i=1}^r g_i(x_i).$$
For each $i \in [r]$, let $\mu_i$ be a probability measure on $H$.
We define the probability measure $\bigotimes_{i=1}^r \mu_i$ on $H^r$ by
$$\left(\bigotimes_{i=1}^r \mu_i\right)(x_1, \ldots, x_r) = \prod_{i=1}^r \mu_i(x_i).$$
\end{definition}

\begin{lemma}
\label{gowerslemma2}
$\| \bigotimes_{i=1}^r g_i \|_{U^d, \bigotimes_{i=1}^r \mu_i} = \prod_{i=1}^r\| g_i \|_{U^d, \mu_i}$.
\end{lemma}
\begin{proof}
Follows by expanding both sides and using the fact
that $\left(\bigotimes_{i=1}^r\mu_i\right)^{(d)} = \bigotimes_{i=1}^r\left(\mu_i^{(d)}\right).$

\end{proof}

\begin{lemma}
\label{gowerslemma3}
Let $q > 1$ be an integer and let $\omega \in \mathbb C$ be a primitive
$q^{\mathrm{th}}$-root of unity. For all $f: \Z_q^n \rightarrow \mathbb C$, all probability
measures $\mu$ on $\Z_q^n$, and
all polynomials $h \in \Z_q[Y_1, \ldots, Y_n]$ of degree $< d$,
$$\|f \omega^h \|_{U^d, \mu} = \| f \|_{U^d, \mu}.$$
\end{lemma}
The above lemma follows from the fact that $(D_{\mathbf t}f) = (D_{\mathbf t} (f\cdot \omega^h))$.

\begin{lemma}
\label{gowerslemma1}
Let $a \in \Z_q\setminus \{0\}$ and let $g : \Z_q^d \rightarrow \mathbb C$  be given by
$g(y) = \omega^{a\prod_{i=1}^d y_i}$.
Let $\mu$ be a probability distribution on $\Z_q^d$ with $\supp(\mu) \supseteq \{0,1\}^d$. Then $\| g \|_{U^d, \mu} < 1 - \epsilon$, where $\epsilon > 0$ depends only on $q, d$ and $\mu$.
\end{lemma}
\begin{proof}
As $\{0,1\} \subseteq \supp(\mu)$, the distribution $\mu^{(d)}$ give some positive probability $\delta > 0$
to the point $(x_0, \mathbf e) = ( x_0, e_1, \ldots, e_d)$, where $x_0 = 0 \in \Z_q^d$, and $e_i \in \Z_q^d$ is the vector with $1$ in the $i$th coordinate
and $0$ in all other coordinates (and $\delta$ depends only on $q, d$ and $\mu$). Then 
$(D_{\mathbf e}g) (x_0) = \prod_{S \subseteq [d]} g(\sum_{i\in S} e_i)^{\circ S} = \omega^{\pm a} \neq 1$ (since whenever $S \neq [d]$,
we have $g(\sum_{i\in S} e_i) = 1$). On the other hand, whenever 
$\mathbf t \in (\Z_q^d)^d$ has some coordinate equal to $0$, which also happens with positive probability 
depending only on $d, \mu$ and $q$, we have $(D_{\mathbf t}g(x)) = 1$. Thus in the
expression
$$\| g\|_{U^d, \mu} = \left|\mathbb E_{(x, \vec t) \sim \mu^{(d)}} \left[(D_{\vec t}f)(x)\right]\right|^{\frac{1}{2^d}},$$
since every term in the expectation has absolute value at most $1$, and we just found two terms with
positive probability with values $1$ and $\omega^{\pm a} \neq 1$, we conclude that $\|g\|_{U^d, \mu} < 1-\epsilon$
for some $\epsilon$ depending only on $q, \mu$ and $d$.
\end{proof}

We now put together the above ingredients.
\begin{theorem}
\label{thm:supergowers}
Let $f: (\Z_q^d)^r \rightarrow \mathbb C$ be given by
$$f(x_1, \ldots, x_r) = \omega^{\sum_{j = 1}^r a_j \prod_{i=1}^d x_{ij}},$$
where $a_j \in \Z_q \setminus \{0\}$ for all $j \in [r]$.
Let $\mu$ be a probability distribution on $\Z_q^d$ with $\supp(\mu)
\supseteq \{0,1\}^d$.
Then for all polynomials $h \in \Z_q[(Y_{ij})_{i \in [d], j \in [r]}]$, with
$\deg(h) < d$, we have
$$\left|\mathbb E_{x \sim \mu^{\otimes r}}\left[ f(x) \omega^{h(x)} \right]\right| \leq c^{r},$$
where $c < 1$ depends only on $q, d$ and $\mu$.
\end{theorem}
\begin{proof}
Let $g_j : \Z_q^d \rightarrow \mathbb C$ be given by $g_j(y) = \omega^{a_j \prod_{i=1}^d y_i}$
(as in in Lemma~\ref{gowerslemma1}), and
take $c = 1 - \epsilon$ from that Lemma.
Notice that $f = \otimes_{j = 1}^r g_j$. Therefore by Lemma~\ref{gowerslemma2},
we have $$\| f \|_{U^d, \mu^{\otimes r}} = \prod_{j = 1}^r \|g_j \|_{U^d, \mu} \leq c^r.$$
As the degree of $h$ is at most $d-1$, Lemma~\ref{gowerslemma3}
implies that $$\| f\omega^h \|_{U^d, \mu^{\otimes r}} = \|f\|_{U^d, \mu^{\otimes r}} \leq c^r.$$
Lemma~\ref{gowerslemma4} now implies that
$$ \left|\mathbb E_{x \sim \mu^{\otimes r}}\left[f(x) \omega^{h(x)} \right]\right| \leq c^r,$$
as desired.
\end{proof}

We can now complete the proof of Lemma~\ref{lem:VW}.

\begin{proofof}{Lemma~\ref{lem:VW}}
By fixing the variables $Z_i$ for $i \not\in \cup_{j} E_j$, and then averaging over
all such fixings, it suffices to consider the case $[m] = \cup_{j} E_j$.
Then the polynomial $Q(Z_1, \ldots, Z_m) = \left(\sum_{j = 1}^{r} a_j \prod_{i \in E_j} Z_i\right) + R(Z)$
can be rewritten in the form (after renaming the variables):
$$\sum_{j=1}^r a_j \prod_{i = 1}^d X_{ij} + h(\mathbf X),$$
where $\deg(h) < d$. Let $\mu$ be the $p$-biased probability measure on $\{0,1\}^d \subseteq \Z_q^d$.
Theorem~\ref{thm:supergowers} now implies that
$$\left|\mathbb E_{x \sim \mu^{\otimes r}} \left[ \omega^{Q(x)} \right] \right| \leq 2^{-\Omega_{q,p,d}(r)},$$
as desired.
\end{proofof}

\section{Proof of Theorem~\ref{thm:Main1}}
\label{sec:quantstatement}

The proof of Theorem~\ref{thm:Main1}
will be via a more general theorem amenable to inductive proof,
Theorem~\ref{thm:foptopoly}.
Just as Theorem~\ref{thm:Main1} states that for almost all $G \in G(n,p)$,
the truth of any $\FOM$ sentence on $G$ is
determined by subgraph frequencies, $\freq^c_G$, Theorem~\ref{thm:foptopoly}
states that for almost all graphs $G \in G(n,p)$, for any
$w_1,\ldots, w_k \in V_G$ the truth of any $\FOM$ {\em formula} $\varphi(w_1, \ldots, w_k)$ on $G$ is
determined by the adjacency and equality information about $w_1, \ldots, w_k$ 
(which we will call the {\em type}), and the {\em labelled subgraph
frequencies at $\mathbf w$}. In the next subsection, we formalize
these notions.

\subsection{Labelled graphs and labelled subgraph frequencies}

Let $I$ be a finite set. We begin with some preliminaries on $I$-labelled graphs.

\begin{definition}[$I$-labelled graphs]
An {\em $I$-labelled graph} is a graph $F = (V_F, E_F)$ where some vertices
are labelled by elements of $I$, such that {\it (a)}
for each $i \in I$, there is exactly one vertex labelled $i$. We denote this
vertex $F(i)$, and {\it (b)} the graph induced on the set of labelled vertices is an independent set.
We denote the set of labelled vertices of $F$ by $\mathcal L(F)$.
\end{definition}

\begin{definition}[Homomorphisms and Copies]
A {\em homomorphism} from an $I$-labelled graph $F$ to a pair $(G,\mathbf w)$,
where $G$ is a graph and $\mathbf w \in V_G^I$, is a homomorphism $\chi\in\Hom(F,G)$
such that for each $i\in I$, $\chi$ maps $F(i)$ to $w_i$.
A homomorphism from $F$ to $(G, \mathbf w)$ is called {\em injective} if for any distinct
$v,w \in V_F$, such that $\{v,w\} \not\subseteq \mathcal L(F)$, we have
$\chi(v) \neq \chi(w)$. A {\em copy} of $F$ in $(G,\mathbf w)$ is a set $E \subseteq E_G$
such that there exists an injective homomorphism $\chi$ from $F$ to $(G,\mathbf w)$ such that
$E = \chi(E_F) := \{(\chi(v), \chi(w)) \mid (v, w) \in E_F\}$. An {\em automorphism} of $F$
is an injective homomorphism from $F$ to $(F, \mathbf w)$, where $w_i = F(i)$ for each $i \in I$.
\end{definition}

\begin{definition}[$\Hom, \Inj, \Cop, \Aut$ for labelled graphs]
Let $F$ be an $I$-labelled graph, and $G$ be any graph. Let
$\mathbf w \in V_G^I$. We define $\Hom(F,(G,\mathbf w))$
to be the set of homomorphisms from $F$ to $(G, \mathbf w)$.
We define $\Inj(F, (G,\mathbf w))$ to be the set of injective homomorphisms
from $F$ to $(G,\mathbf w)$. We define $\Cop(F, (G, \mathbf w))$ to be the set of
copies of $F$ in $(G, \mathbf w)$. We define $\Aut(F)$ to be the set of automorphisms
of $F$.
We let $[F](G, \mathbf w)$ (respectively $\langle F \rangle(G, \mathbf w)$, $\aut(F)$)
be the cardinality of $\Inj(F,(G,\mathbf w))$ (respectively $\Cop(F,(G,\mathbf w))$, $\Aut(F)$).

Finally, let $[F]_q(G,\mathbf w) = [F](G, \mathbf w) \mod q$ and $\langle F \rangle_q(G, \mathbf w) = \langle F \rangle(G, \mathbf w) \mod q$.
\end{definition}

\begin{definition}[Label-connected]
For $F$ an $I$-labelled graph, we say $F$ is {\em \konn} if
$F \setminus \mathcal L(F)$ is connected.
Define $\Conn_I^{t}$ to be the set of all $I$-labelled {\konn} graphs with at most
$t$ unlabelled vertices. For $i \in I$, we say an $I$-labelled graph $F$ is
{\bf dependent on label $i$} if $F(i)$ is not an isolated vertex.
\end{definition}

\begin{definition}[Partitions]
If $I$ is a set, an {\em $I$-partition} is a set of subsets
of $I$ that are pairwise disjoint, and whose union is $I$.
If $\Pi$ is an $I$ partition, then for $i \in I$ we denote the unique
element of $\Pi$ containing $i$ by $\Pi(i)$. If $V$ is any set and $\mathbf w \in V^I$,
we say $\mathbf w$ respects $\Pi$ if for all $i, i' \in I$, 
$w_i = w_{i'}$ iff $\Pi(i) = \Pi(i')$.
\end{definition}
The collection of all partitions of $I$ is denoted
$\Partitions(I)$.

If $I \subseteq J$, $\Pi \in \Partitions(I)$ and $\Pi' \in \Partitions(J)$,
we say $\Pi'$ extends $\Pi$ if for all $i_1, i_2 \in I$, $\Pi(i_1) = \Pi(i_2)$ if
and only if $\Pi'(i_1) = \Pi'(i_2)$.

\begin{definition}[Types]
An $I$-$\type$ $\tau$ is a pair $(\Pi_\tau, E_\tau)$ where
$\Pi_\tau \in \Partitions(I)$ and $E_\tau \subseteq {\Pi_{\tau}\choose 2}$.
For a graph $G$ and $\mathbf w \in V_G^I$, we define
the $\type$ of $\mathbf w$ in $G$, denoted $\type_G(\mathbf w)$, to be
the $I$-$\type$ $\tau$, where $\mathbf w$ respects $\Pi_{\tau}$,
and for all $i, i' \in I$, $\{\Pi_{\tau}(i), \Pi_{\tau}(i')\} \in E_{\tau}$ if and only if $w_i$ and $w_{i'}$ are adjacent in $G$.
\end{definition}
The collection of all $I$-$\type$s is denoted $\Types(I)$.

If $I \subseteq J$, and $\tau \in \Types(I)$ and $\tau' \in \Types(J)$, we
say $\tau'$ {\em extends} $\tau$ if $\Pi_{\tau'}$ extends $\Pi_{\tau}$ and
for each $i_1, i_2 \in I$,
$\{\Pi_\tau(i_1), \Pi_\tau(i_2)\} \in E_{\tau}$ if and only
if $\{\Pi_{\tau'}(i_1), \Pi_{\tau'}(i_2)\} \in E_{\tau'}$.

\begin{definition}[Labelled subgraph frequency vector]
Let $G$ be a graph and $I$ be any set.
Let $\mathbf w \in V_G^I$. We define the {\em labelled
subgraph frequency vector at $\mathbf w$},
$\freq_{G}^a(\mathbf w) \in \Z_q^{\Conn_I^{a}}$, to
be the vector such that for each $F \in \Conn_I^{a}$,
$$(\freq_G^a(\mathbf w))_F = [F]_q(G, \mathbf w).$$
\end{definition}

\begin{remark} We will often deal with $[k]$-labelled graphs. By abuse of notation we will refer to them as $k$-labelled graphs. If $\mathbf w \in V^{[k]}$ and $u \in V$, when we refer to the tuple $(\mathbf w, v)$, we mean the $[k+1]$-tuple whose first $k$ coordinates are given by $\mathbf w$ and whose $k+1$st coordinate is $v$. Abusing notation even further, when we deal with a $[k+1]$-labelled graph $F$, then by $[F](G, \mathbf w, v)$, we mean $[F](G, (\mathbf w, v))$. Similarly $\Conn_k^{t}$ denotes $\Conn_{[k]}^t$.
\end{remark}

\subsection{The quantifier eliminating theorem}

We now state Theorem~\ref{thm:foptopoly}, from which Theorem~\ref{thm:Main1}
follows easily. Informally, it says that an $\FOM$-formula $\varphi(\mathbf w)$
is essentially determined by the type of $\mathbf w$, $\type_G(\mathbf w)$, and
the labelled subgraph frequencies at $\mathbf w$, $\freq_G^c(\mathbf w)$.

\begin{theorem}
\label{thm:foptopoly}
For all primes $q$ and integers $k, t > 0$, there is a constant $c = c(k, t, q)$ such that for every $\FOM$ formula $\varphi(\alpha_1, \ldots, \alpha_k)$ with quantifier depth $t$, there is a function
$\psi : \Types(k) \times \Z_q^{\Conn_{k}^{c}} \rightarrow \{0, 1\}$ such that for all $p \in (0,1)$, the quantity
\begin{align*}
\Pr_{G \in G(n,p)}
\left[
\begin{array}{c}
\forall w_1, \ldots, w_k \in V_G, \quad\quad\quad\quad\quad\quad\quad\quad\quad\quad\quad\quad\quad\quad\quad\quad\quad\quad\\
\quad\quad(G \mbox{ satisfies }\varphi(w_1, \ldots, w_k)) \Leftrightarrow (\psi(\type_G(\mathbf w), \freq^c_G(\mathbf w)) = 1)
\end{array}
\right]
\geq 1 - 2^{-\Omega(n)}.
\end{align*}
\end{theorem}

Putting $k = 0$, we recover Theorem~\ref{thm:Main1}.

We now give a brief sketch of the proof of Theorem~\ref{thm:foptopoly} (the detailed proof appears in Section~\ref{sec:quant}). The proof is by induction on the size of the formula $\varphi$. When the formula
$\varphi$ has no quantifiers, then the truth of $\varphi(\mathbf w)$ on $G$ is completely
determined by $\type_G(\mathbf w)$. The case where $\varphi$ is of the form
$\varphi_1(\alpha_1, \ldots, \alpha_k) \wedge \varphi_2(\alpha_1, \ldots, \alpha_k)$ is
easily handled via the induction hypothesis. The case where
$\varphi(\alpha_1, \ldots, \alpha_k) = \neg \varphi_1(\alpha_1, \ldots, \alpha_k)$ is similar.

The key cases for us to handle are thus {\it (i)} $\varphi(\alpha_1, \ldots, \alpha_k)$
is of the form $\Mod_q^i \beta, \varphi'(\alpha_1, \ldots, \alpha_k, \beta)$, and {\it (ii)}
$\varphi(\alpha_1, \ldots, \alpha_k)$ is of the form $\exists \beta, \varphi'(\alpha_1, \ldots, \alpha_k, \beta)$.
We now give a sketch of how these cases may be handled.

For case (i), let $\psi' : \Types(k+1) \times \Z_q^{\Conn_{k+1}^b}$ be the function
given by the induction hypothesis for the formula $\varphi'$.
Thus for most graphs $G \in G(n,p)$ (namely the ones for which $\psi'$ is good
for $\varphi'$), $\varphi(w_1, \ldots, w_k)$
is true if and only the number of vertices $v \in V_G$ 
such that 
$\psi'(\type_{G}(\mathbf w, v), \freq_G^b(\mathbf w, v)) = 1$ is congruent to $i$ mod $q$. In Theorem~\ref{thm:extendcount} (whose proof appears in Section~\ref{sec:extendcount}), we show that the number of such vertices
$v$ can be determined solely as a function of $\type_G(\mathbf w)$
and $\freq^a_G(\mathbf w)$ for suitable $a$.
This fact allows us to define $\psi$ in a natural way,
and this completes case (i).

Case (ii) is the most technically involved case. As before, we get a function
$\psi'$ corresponding to $\varphi'$ by the induction hypothesis. We show
that one can define $\psi$ essentially as follows: define $\psi(\tau, f) =1$
if there exists some $(\tau', f') \in \Types(k+1) \times \Z_q^{\Conn_{k+1}^b}$
that ``extends" $(\tau, f)$ for which $\psi'(\tau', f') = 1$; otherwise
$\psi(\tau, f) = 0$. Informally, we show that if it is conceivable that there is a vertex
$v$ such that $\varphi'(\mathbf w, v)$ is true, then $\varphi(\mathbf w)$ is
almost surely true. Proving this statement requires us to
get a characterization of the distribution of labelled subgraph frequencies, significantly generalizing Theorem~\ref{thm:Main2}. This is done in Theorem~\ref{thm:superindep} (whose proof appears in Section~\ref{sec:indep}).


\section{Quantifier Elimination}
\label{sec:quant}

In this section, we give a full proof of Theorem~\ref{thm:foptopoly}.
Before doing so, we state the main technical theorems: Theorem~\ref{thm:extendcount}
(which is needed for eliminating $\Mod_q$ quantifiers), and Theorem~\ref{thm:superindep}
(which is needed for eliminating $\exists$ quantifiers). We do this in the following
two subsections.

\subsection{Counting extensions}

The next theorem plays a crucial role in the elimination
of the $\Mod_q$ quantifiers. This is the only step where the assumption
that $q$ is a prime plays a role in the modular convergence law.

\begin{theorem}
\label{thm:extendcount}
Let $q$ be a prime, let $k, b >0 $ be integers and let $a \geq (q-1)\cdot b \cdot|\Conn_{k+1}^{b}|$.
There is a function
$$\lambda: \Types(k+1) \times \Z_q^{\Conn_{k+1}^{b}} \times \Types(k) \times \Z_q^{\Conn_k^{a}} \rightarrow \Z_q$$
such that for  all $\tau' \in \Types(k+1)$, $f' \in \Z_q^{\Conn_{k+1}^{b}}$,  $\tau \in \Types(k)$, $f \in \Z_q^{\Conn_k^a}$, it holds that for every graph $G$, and every $w_1, \ldots, w_k \in V_G$ with $\type_G(\mathbf w) = \tau$ and $\freq_G^{a}(\mathbf w) = f$, the cardinality of the set
$$\{v \in V_G : \type_G(\mathbf w, v) = \tau' \wedge \freq_G^{b}(\mathbf w, v) = f'\}$$
is congruent to $\lambda(\tau', f', \tau, f) \mod q.$
\end{theorem}

The proof appears in Section~\ref{sec:extendcount}.
The principal ingredient in its proof is the following lemma, which states
that the numbers $[F](G, \mathbf w)$, as $F$ varies over small label-connected graphs, determine the number $[F'](G, \mathbf w)$ for all small graphs $F'$.

\begin{lemma}{\bf (Label-connected subgraph frequencies determine all subgraph frequencies)}
\label{lem:kgraphtoconnected}
For every $k$-labelled graph $F'$ with $|V_{F'} \setminus \mathcal L(F')| \leq t$, there is a polynomial $\delta_{F'} \in \mathbb Z[(X_F)_{F \in \Conn_k^{t}}]$ such that for all graphs $G$ and $\mathbf w \in V_G^k$,
$$[F'](G, \mathbf w) = \delta_{F'}(x),$$
where $x \in \mathbb Z^{\Conn_k^{t}}$ is given by $x_{F} = [F](G, \mathbf w)$.
\end{lemma}

\subsection{The distribution of labelled subgraph frequencies mod q}

In this subsection, we state the theorem that will help us eliminate $\exists$ quantifiers.
Let us first give an informal description of the theorem.
We are given a tuple $\mathbf w \in [n]^k$, and distinct
$u_1, \ldots, u_s \in [n] \setminus \{w_1, \ldots, w_k\}$.
Let $G$ be sampled from $G(n,p)$ (recall that we think of $G(n,p)$
as a random graph whose vertex set is $[n]$: thus the $w_i$ and $u_j$
are vertices of $G$).
The theorem completely describes the joint distribution of the labelled subgraph frequency vectors
at all the tuples $\mathbf w$, $(\mathbf w, u_1), \ldots, (\mathbf w, u_s)$; namely
it pins down the distribution of
$(\freq_G^a(\mathbf w), \freq_G^b(\mathbf w, u_1), \ldots, \freq_G^b(\mathbf w, u_s))$.
We first give a suitable definition of the set of {\em feasible frequency vectors},
and then claim that (a) the $\freq_G^a(\mathbf w)$ is essentially uniformly distributed
over the set of its feasible frequency vectors, and (b) conditioned on $\freq_G^a(\mathbf w)$,
the distributions of $\freq_G^b(\mathbf w, u_1), \ldots, \freq_G^b(\mathbf w, u_s)$ are all
essentially independent and uniformly distributed over the set of those feasible frequency
vectors that are ``consistent" with $\freq_G^a$.

To define the set of feasible frequency vectors (which will equal the set of all possible
values that $\freq_G^a(\mathbf w)$ may assume), there are two factors that come into play.
The first factor, one that we already encountered while dealing with unlabelled graphs,
is a divisibility constraint: the number $[F](G, \mathbf w)$ is always divisible
by a certain integer depending on $F$, and hence for some $F$, it cannot assume arbitrary values mod $q$. The second factor is a bit subtler: when $w_1, \ldots, w_k$ are not all distinct,
for certain pairs $F, F'$ of label-connected $k$-labelled graphs,
$[F](G, \mathbf w)$ is forced to equal $[F'](G,\mathbf w)$.
Let us see a simple example of such a phenomenon. Let $k = 2$ and let $w_1 = w_2$. 
Let the $2$-labelled graph $F$ be a path of length $2$ with ends labelled 1 and 2.
Let the $2$-labelled graph $F'$ be the disjoint union of an edge, one of whose ends is labelled 1, and an isolated vertex labelled 2. Then in any graph $G$, $[F](G, \mathbf w) = [F'](G, \mathbf w) = $
the degree of $w_1$. 

In the rest of this subsection, we will build up some notation and results leading up to a definition of feasible frequency vectors
and the statement of the main technical theorem describing the distribution of labelled subgraph
frequency vectors.

\begin{definition}[Quotient of a labelled graph by a paritition]
Let $F$ be a $I$-labelled graph and let $\Pi \in \Partitions(I)$. We define $F/\Pi$ to be the $\Pi$-labelled graph
obtained from $F$ by (a) for each $J \in \Pi$, identifying all the vertices with
labels in $J$ and labelling this new vertex $J$, and (b) deleting duplicate edges.
If $F$ and $F'$
are $I$-labelled graphs and $\Pi \in \Partitions(I)$, we say $F$
and $F'$ are $\Pi$-equivalent if $F/\Pi \cong F'/\Pi$.
\end{definition}

Let $\mathbf w \in V_G^{I}$. Let $\Pi \in \Partitions(I)$
be such that $\mathbf w$ respects $\Pi$. Define 
$(\mathbf w / \Pi) \in V_G^{\Pi}$ by: for each $J \in \Pi$, $(\mathbf w/\Pi)_J = w_j$,  where $j$ is any element
of $J$ (this definition is independent of the
choice of $j \in J$). Observe that as $J$ varies
over $\Pi$, the vertices $(w/\Pi)_J$ are all distinct.


The next two lemmas show that the numbers $[F](G, \mathbf w)$
must satisfy certain constraints. These constraints will eventually
motivate our definition of feasible frequency vectors.

\begin{lemma}
\label{lem:partitionedinj}
If $G$ is a graph and $\mathbf w \in V_G^I$, with
$\mathbf w$ respecting $\Pi \in \Partitions(I)$,
then for any $I$-labelled $F$,
\begin{align}
[F](G, \mathbf w) = [F/\Pi](G, (\mathbf w/\Pi)).
\end{align}
\end{lemma}
\begin{proof}
We define a bijection $\alpha: \Inj(F/\Pi, (G, \mathbf w/\Pi)) \rightarrow \Inj(F, (G, \mathbf w))$.
Let $\pi \in \Hom(F, F/\Pi)$ be the natural homomorphism sending each unlabelled vertex in $V_F$ to its corresponding vertex in $V_{F/\Pi}$, and, for each $i \in I$ sending $F(i)$ to $(F/\Pi)(\Pi(i))$.
We define $\alpha(\chi)$ to be $\chi\circ \pi$.

Take distinct $\chi, \chi'\in \Inj(F/\Pi, (G, \mathbf w/\Pi))$. Let $u \in V_{F/\Pi}$ with $\chi(u) \neq \chi'(u)$. Note that $u$ cannot be an element of $\mathcal L(F/\Pi)$, for if $u = (F/\Pi)(\Pi(i))$, then $\chi(u) = \chi'(u) = w_i$. Thus $u \not\in \mathcal L(F/\Pi)$. Let $v \in V_F$ be the vertex $\pi^{-1}(u)$ (which is uniquely specified since $u \not\in \mathcal L(F/\Pi)$). Thus we have $\chi(\pi(v)) = \chi(u) \neq \chi'(u) = \chi'(\pi(v))$. Thus $\alpha(\chi) \neq \alpha(\chi')$, and $\alpha$ is one-to-one.

To show that $\alpha$ is onto, take any $\chi \in \Inj(F, (G, \mathbf w))$. Define $\chi' \in \Inj(F/\Pi, (G, \mathbf w/\Pi))$ by:
\begin{enumerate}
\item $\chi'(u) = \chi(\pi^{-1}(u))$ if $u \not\in \mathcal L(F/\Pi)$.
\item $\chi'(u) = w_j$ for any $j \in J$, if $u = (F/\Pi)(J)$ with $J \in \Pi$.
\end{enumerate}
Then $\alpha(\chi') = \chi$.
\end{proof}

\begin{lemma}
\label{lem:aut}
Let $G$ be a graph and $\mathbf w \in V_G^{I}$. Suppose
all the $(w_i)_{i \in I}$ are distinct.
Let $F$ be an $I$-labelled \konn graph with $|E_F| \geq 1$.
Then
$$ [F](G, \mathbf w) = \aut(F) \cdot \langle F \rangle(G, \mathbf w).$$
\end{lemma}
\begin{proof}
We give a bijection $\alpha:  \Aut(F) \times \Cop(F, (G, \mathbf w)) \rightarrow \Inj(F, (G, \mathbf w))$.

For each $E \in \Cop(F, (G, \mathbf w))$, we fix a $\chi_E \in \Inj(F, (G, \mathbf w))$ such that
$\chi_E(E_F) = E$.
Then we define $\alpha(\sigma, E) = \chi_E \circ \sigma$.

First notice that $\alpha(\sigma, E)(E_F) = \chi_E(\sigma (E_F)) = \chi_E (E_F) = E$.
Thus if $\alpha(\sigma,E)  = \alpha(\sigma',E')$, then $E = E'$.
But since $\chi_E$ is injective, for any $\sigma \neq \sigma'$, we have
$\chi_E \circ \sigma \neq \chi_E \circ \sigma'$. Thus $\alpha$ is one-to-one.

To show that $\alpha$ is onto, take any $\chi \in \Inj(F, (G, \mathbf w))$.
Let $E = \chi(E_F)$. As $F$ is \konn and $\chi_E(E_F) = \chi(E_F)$,
we have $\chi_E(V_F) = \chi(V_F)$. We may now define
$\sigma \in \Aut(F)$ by $\sigma(u) = \chi_E^{-1}(\chi(u))$ for
each $u \in V_F$.
Clearly, $\alpha(\sigma, E) = \chi$, and so $\alpha$ is onto.

Thus $\alpha$ is a bijection, and the lemma follows.
\end{proof}

Note that Lemma~\ref{lem:aut-ident} and Lemma~\ref{lem:autunlab} follow formally from the above
lemma.


Let $\K_1(I)$ be the $I$-labelled graph with $|I| + 1$ vertices:
$|I|$ labelled vertices and one isolated unlabelled vertex. The role of $\K_1(I)$
in the $I$-labelled theory is similar to the role of $\K_1$ in the 
unlabelled case.

\begin{definition}[Feasible frequency vectors]
We define the set of {\em feasible frequency vectors},
$\Freq(\tau, I, a)$ to be the set of $f \in \Z_q^{\Conn_I^a}$
such that
\begin{enumerate}
\item[(a)] for any $F \in \Conn_I^a$, we have $f_F \in \aut(F / \Pi_{\tau}) \cdot \Z_q$.
\item[(b)] for any $F, F' \in \Conn_I^a$ that are $\Pi_{\tau}$-equivalent,
we have $f_F = f_{F'}$.
\end{enumerate}
Let $\Freq_{n}(\tau, I, a)$ be the set $\{ f \in \Freq(\tau, I, a) : f_{K_1(I)} = n - |\Pi_{\tau}| \mod q \}$.
Note that if $n = n' \mod q$, then $\Freq_n(\tau, I, a) = \Freq_{n'}(\tau, I, a)$.
\end{definition}

Observe that for any $\mathbf w \in V_G^I$ with $\type_G(\mathbf w) = \tau$, the vector $\freq_G^a(\mathbf w)$ is an element of $\Freq(\tau, I, a)$. This follows from Lemma~\ref{lem:partitionedinj} and Lemma~\ref{lem:aut},
which allow us to deduce (recall that $(\mathbf w/\Pi_{\tau})_J$ are all distinct for $J \in \Pi_{\tau}$) that for any $F \in \Conn_I^a$,
\begin{equation}
\label{eq:partitionedcop}
[F](G, \mathbf w) = \aut(F/\Pi_{\tau}) \cdot \langle F/\Pi_{\tau}\rangle(G, \mathbf w/\Pi_{\tau}).
\end{equation}
Observe also that if $|V_G| = n$, then $\freq_G^a(\mathbf w) \in \Freq_n(\tau, I, a)$, since
$[K_1(I)](G, \mathbf w) = |V_G \setminus \{w_1, \ldots, w_k\}| = n - |\Pi_{\type(\mathbf w)}|$, as
required by the definition.

\begin{definition}[Extending]
Let $I$ be a set and let $J = I \cup \{i^*\}$. Let $a \geq b > 0$ be positive integers.
We say $(\tau', f') \in \Types(J) \times \Freq(\tau', J, b)$
{\em extends} $(\tau, f) \in \Types(I) \times \Freq(\tau, I, a)$ if $\tau'$ extends $\tau$,
and for every $F \in \Conn_I^b$, we have
\begin{enumerate}
\item if $\{i^*\} \not\in \Pi_{\tau'}$,
\begin{equation}
\label{eq:self}
f_F = f'_{\widetilde{F}},
\end{equation}
where $\widetilde{F}$ is the graph obtained from $F$ by introducing an isolated vertex labelled $i^*$.
\item if $\{i^*\} \in \Pi_{\tau'}$, letting $\delta_H :\Z_q^{\Conn_J^b} \rightarrow \Z_q$ be
the function given by Lemma~\ref{lem:kgraphtoconnected},
\begin{equation}
\label{eq:freqextend}
f_F = f'_{\widetilde{F}} + \sum_{u \in V_F\setminus \mathcal L(F)} c_u \delta_{F_u}(f'),
\end{equation}
where
\begin{itemize}
\item $\widetilde{F}$ is the graph obtained from $F$ by introducing an isolated vertex labelled $i^*$.
\item $c_u$ equals $1$ if for all $i \in I$, if $u$ is adjacent to $F(i)$,
then $\{\Pi_{\tau'}(i^*), \Pi_{\tau'}(i)\} \in E_{\tau'}$. Otherwise, $c_u = 0$.
\item $F_u$ is the graph obtained from $F$ by labelling the vertex $u$ by $i^*$ and deleting all
edges between $u$ and the other labelled vertices of $F$.
\end{itemize}
\end{enumerate}
\end{definition}

The crux of the above definition is captured in the following lemma.
\begin{lemma}
\label{lem:ffreqcrux}
Let $G$ be a graph. Let $a \geq b> 0$ be integers.
Let $\mathbf w \in V^k$ and $v \in V$.
Let $\tau = \type_G(\mathbf w)$, $\tau' = \type_G(\mathbf w, v)$, $f = \freq_G^a(\mathbf w)$ and
$f' = \freq_G^b(\mathbf w, v)$.
Then $(\tau', f')$ extends $(\tau, f)$.
\end{lemma}
\begin{proof}
We keep the notation of the previous definition. First observe that $\tau'$ extends $\tau$.

If $\{k+1\} \not\in \Pi_{\tau'}$, then we need to show that
$[F]_q(G, \mathbf w) = [\widetilde{F}]_q(G, \mathbf w, v)$ for each $F \in \Conn_{k}^b$.
This is immediate from the definitions.

If $\{k+1\} \in \Pi_{\tau'}$, then we need to show that
$[F]_q(G, \mathbf w) = [\widetilde{F}]_q(G, \mathbf w, v) + \sum_{u \in V_F\setminus \mathcal L(F)} c_u [F_u]_q(G, \mathbf w, v)$.
We do this by counting the $\chi \in \Inj(F, (G,\mathbf w))$ based on its image $\chi(V_F)$
as follows:
\begin{enumerate}
\item Category 1: $v \not\in \chi(V_F)$. There are precisely $[\widetilde{F}](G,\mathbf w, v)$ such $\chi$.
\item Category 2: $v = \chi(u)$ (in this case $u$ is uniquely specified).
Note that $u \not\in \mathcal L(F)$. Then it must be the case that for any $i \in [k]$ such
that $u$ is adjacent to $F(i)$, $w_i$ is adjacent to $v$. Thus
$\{\Pi_{\tau'}(i), \Pi_{\tau'}(k+1)\} \in E_{\tau'}$, and so $c_u = 1$. The number of
such $\chi$ is $[F_u](G,\mathbf w, v)$.
\end{enumerate}
This proves the desired relation.
\end{proof}

We now state and prove two key uniqueness properties enjoyed by the notion of
extension.

\begin{lemma}
\label{lem:unique-extend}
Let $a \geq b > 0$ be integers.
Let $\mathbf w \in V_G^k$. Let $u \in V_G \setminus \{ w_1, \ldots, w_k\}$.
Let $\tau = \type_G(\mathbf w)$ and $\tau' = \type_G(\mathbf w, u)$.
Let $f = \freq_G^a(\mathbf w)$. Then $\freq_{G}^b(\mathbf w, u)$ is the 
unique $f' \in \Z_q^{\Conn_{k+1}^b}$ such that:
\begin{itemize}
\item for each $H \in \Conn_{k+1}^b$ that is dependent on label $k+1$,
we have $f'_{H} = [H]_q(G,\mathbf w, u)$.
\item $(\tau', f')$ extends $(\tau,f)$.
\end{itemize}
\end{lemma}
\begin{proof}
By Lemma~\ref{lem:ffreqcrux}, the vector $\freq_G^b(\mathbf w, u)$ is such an $f'$.

To prove uniqueness, it suffices to show that any $f'$ satisying these two properties
equals $\freq_G^b(\mathbf w, u)$. Thus it suffices to show that for any $H \in \Conn_{k+1}^b$ not dependent
on label $k+1$, $f'_H = (\freq_{G}^{b}(\mathbf w, u))_H$.

We prove this by induction on $|V_H \setminus \mathcal L(H)|$.
Let $H \in \Conn_{k+1}^b$ not dependent on label $k+1$. Thus
$H$ is of the form $\widetilde{F}$ for some graph $F \in \Conn_{k}^b$
(as in the previous lemma, for a $[k]$-labelled graph $F$,
we let $\widetilde{F}$ be the $[k+1]$-labelled graph obtained by
adjoining an isolated vertex labelled $k+1$ to $F$).
By Equation~\eqref{eq:freqextend}, we see that $f'_H$ is
{\em uniquely} determined by $\tau$, $\tau'$,
$f_F$ and the numbers $(f'_{H'})_{H' \in \Conn_{k+1}^{|V_H \setminus \mathcal L(H)| - 1}}$ 
(since each $c_u$ is determined by $\tau'$ and each of the graphs $F_u$ have
$|F_u \setminus \mathcal L(F_u)| \leq |V_H \setminus \mathcal L(H)| - 1$). By
induction hypothesis, all the $f'_{H'} = (\freq_G^b(\mathbf w, u))_{H'}$. Thus,
since $\freq_G^b(\mathbf w, u)$ also satisfies Equation \eqref{eq:freqextend},
we have $f'_H = (\freq_G^b(\mathbf w, u))_H$, as required.
\end{proof}

\begin{lemma}
\label{lem:self-fulfill}
Let $a \geq b > 0$ be integers.
Let $(\tau, f) \in \Types(k) \times \Freq(\tau, [k], a)$.
Let $\tau' \in \Types(k+1)$ extend $\tau$ with $\{k+1\} \not\in \Pi_{\tau'}$.
Then there is at most one $f' \in \Freq(\tau', [k+1], b)$ such that $(\tau', f')$ extends $(\tau, f)$.
\end{lemma}
\begin{proof}
As in the previous lemma, for a $[k]$-labelled graph $F$,
we let $\widetilde{F}$ be the $[k+1]$-labelled graph obtained by
adjoining an isolated vertex labelled $k+1$ to $F$.
For any $F \in \Conn_{k}^b$, we must have $f'_{\widetilde{F}} = f_{F}$.
Now we claim that any $H \in \Conn_{k}^b$ is $\Pi$-equivalent to some
graph of the form $\widetilde{F}$. To prove this, let $j \in [k]$ be such that
$\Pi_{\tau'}(j) = \Pi_{\tau'}(k+1)$. Let $H^*$ be the graph obtained from $H$ by
adding, for each neighbor $u$ of $H(k+1)$, an edge between $u$ and
the $H(j)$, and then removing (a) all edges incident on $H(k+1)$, and (b) any duplicate edges
introduced.
By construction, $H/\Pi_{\tau'} \cong H^*/\Pi_{\tau'}$, and so $f'_{H} = f'_{H^*}$ by Equation
~\eqref{eq:partitionedcop}.
In addition, the $H^*(k+1)$ is isolated, and hence $H^*$ is of the form
$\widetilde{F}$ for some $F \in \Conn_k^{b}$.

What we have shown is that for every $H \in \Conn_{k+1}^b$, $f'_{H}$ is forced to equal
$f_F$ for some $F \in \Conn_{k}^{b}$. This implies that $f'$ is specified uniquely.
\end{proof}

Finally, we will need to deal with random graphs $G(n,p)$
with some of the edges already exposed. The next definition captures this object.

\begin{definition}[Conditioned Random Graph]
Let $A = (V_A, E_A)$ be a graph with $V_A \subseteq [n]$.
We define the {\em conditioned random graph} $G(n,p\mid V_A, E_A)$ to be the graph $G = (V_G, E_G)$ with
$V_G = [n]$ and  $E_G = E_A \cup E'$, where each $\{i,j\} \in {[n] \choose 2}\setminus {V_A \choose 2}$ is
included in $E'$ independently with probability $p$.
\end{definition}

We can now state the main technical theorem
that describes the distribution of labelled subgraph
frequencies, and will eventually be useful for eliminating
$\exists$ quantifiers.

\begin{theorem}
\label{thm:superindep}
Let $a \geq b$ be positive integers.
Let $A$ be a graph with $V_A \subseteq [n]$ and $|V_A| \leq n' \leq n/2$. Let $G \in G(n,p \mid V_A, E_A)$. Let $\mathbf w = (w_1, \ldots, w_k) \in V_A^k$, and let $u_1, \ldots, u_s \in V_A\setminus\{w_1, \ldots, w_k\}$ be distinct. Let $\tau = \type_G(\mathbf w)$ and let $\tau_i = \type_G(\mathbf w, u_i)$ (note
that $\tau, \tau_1,\ldots, \tau_s$ are already determined by $E_A$).
Let $f$ denote the random variable $\freq_{G}^a(\mathbf w)$.
Let $f_i$ denote the random variable $\freq_{G}^b(\mathbf w, u_i)$.

Then, there exists a constant $\rho = \rho(a, q, p) > 0$, such that if
$s \leq \rho \cdot n$, then the distribution of $(f, f_1, \ldots, f_s)$ over $\Freq_n(\tau, [k], a) \times \prod_{i} \Freq_n(\tau_i, [k+1], b)$
is $2^{-\Omega(n)}$-close to the distribution of $(h, h_1, \ldots, h_s)$ generated as follows:
\begin{enumerate}
\item $h$ is picked uniformly at random from $\Freq_n(\tau, [k], a)$.
\item For each $i$, each $h_i$ is picked independently and uniformly from the set of all $f' \in \Freq_n(\tau_i, [k+1], b)$
such that $(\tau_i, f')$ extends $(\tau, h)$.
\end{enumerate}
\end{theorem}

\subsection{Proof of Theorem~\ref{thm:foptopoly}}

We now prove Theorem~\ref{thm:foptopoly}, where the main quantifier
elimination step is carried out.

{\bf Theorem~\ref{thm:foptopoly} (restated)} {\it
For every prime $q$ and integers $k, t > 0$, there is a constant $c = c(k, t, q)$ such that for every $\FOM$ formula $\varphi(\alpha_1, \ldots, \alpha_k)$ with quantifier depth $t$, there is a function
$\psi : \Types(k) \times \Z_q^{\Conn_{k}^{c}} \rightarrow \{0, 1\}$ such that for all $p \in (0,1)$, the quantity
\begin{align*}
\Pr_{G \in G(n,p)}
\left[
\begin{array}{c}
\forall w_1, \ldots, w_k \in V_G, \quad\quad\quad\quad\quad\quad\quad\quad\quad\quad\quad\quad\quad\quad\quad\quad\quad\quad\\
\quad\quad(G \mbox{ satisfies }\varphi(w_1, \ldots, w_k)) \Leftrightarrow (\psi(\type_G(\mathbf w), \freq^c_G(\mathbf w)) = 1)
\end{array}
\right]
\geq 1 - 2^{-\Omega(n)}.
\end{align*}
}
\begin{proof}
The proof is by induction on the size of the formula.
If $\varphi(w_1, \ldots, w_k)$ is an atomic formula, then trivially there exists a $\psi: \Types(k) \rightarrow \{0,1\}$ such that for every graph $G$ and every $\mathbf w \in V_G^k$, the statement $\varphi(w_1, \ldots, w_k)$ holds if and only if $\psi(\type_G(\mathbf w)) = 1$. Thus we may take $c(k, 0, q) = 0$. We will show that one may take
$c(k, t, q) = (q-1)c(k+1, t-1, q) \cdot 2^{c(k+1, t-1, q)^2}$.

Now assume the result holds for all formulae smaller than $\varphi$.

{\bf Case $\wedge$:} Suppose $\varphi(\alpha_1, \ldots, \alpha_k) = \varphi_1(\alpha_1, \ldots, \alpha_k) \wedge \varphi_2(\alpha_1, \ldots, \alpha_k)$. By induction hypothesis,
we have functions $\psi_1, \psi_2$ and a constant $c$ such that
$\Pr_{G}[\forall w_1, \ldots, w_k \in V_G, (\varphi_1(w_1, \ldots, w_k) \Leftrightarrow \psi_1(\type_G(\mathbf w), \freq^c_G(\mathbf w)) = 1) ] \geq 1 -2^{-\Omega(n)}$ and $\Pr_{G}[\forall w_1, \ldots, w_k \in V_G, (\varphi_2(w_1, \ldots, w_k) \Leftrightarrow \psi_2(\type_G(\mathbf w), \freq^c_G(\mathbf w)) = 1) ] \geq 1 -2^{-\Omega(n)} $.
Setting $\psi(\tau, f) = \psi_1(\tau, f) \cdot \psi_2(\tau, f)$, it follows from the union bound that
$$ \Pr_{G}[\forall w_1, \ldots, w_k \in V_G, (\varphi(w_1, \ldots, w_k) \Leftrightarrow \psi(\type_G(\mathbf w), \freq^c_G(\mathbf w)) = 1) ] \geq 1 -2^{-\Omega(n)}.$$

{\bf Case $\neg$:} Suppose $\varphi(\alpha_1, \ldots, \alpha_k) = \neg \varphi'(\alpha_1, \ldots, \alpha_k)$. Let $\psi' : \Types(k) \times \Z_q^{\Conn_{k}^{c}} \rightarrow \{0, 1\}$ be such that
$ \Pr_{G}[\forall w_1, \ldots, w_k \in V_G, (\varphi'(w_1, \ldots, w_k) \Leftrightarrow \psi'(\type_G(\mathbf w), \freq^c_G(\mathbf w)) = 1) ] \geq 1 -2^{-\Omega(n)}.$
Setting $\psi(\tau, f) = 1- \psi'(\tau, f)$, we see that
$$ \Pr_{G}[\forall w_1, \ldots, w_k \in V_G, (\varphi(w_1, \ldots, w_k) \Leftrightarrow \psi(\type_G(\mathbf w), \freq^c_G(\mathbf w)) = 1) ] \geq 1 -2^{-\Omega(n)}.$$

{\bf Case $\Mod^i_q$:}
Suppose $\varphi(\alpha_1, \ldots, \alpha_k) = \Mod_q^i \beta, \varphi'(\alpha_1, \ldots, \alpha_k, \beta)$. Let $c' = c(k+1, t-1, q)$ and let $\psi' : \Types(k+1) \times \Z_q^{\Conn_{k+1}^{{c'}}} \rightarrow \{0, 1\}$ be given by the induction hypothesis, so that
$$ \Pr_{G}[\forall w_1, \ldots, w_k, v \in V_G, (\varphi'(w_1, \ldots, w_k, v) \Leftrightarrow \psi'(\type_G(\mathbf w, v), \freq^{c'}_G(\mathbf w, v)) = 1) ] \geq 1 -2^{-\Omega(n)}.$$

Call $G$ {\em good} if this event occurs, i.e., if
$$\forall w_1, \ldots, w_k, v \in V_G, (\varphi'(w_1, \ldots, w_k, v) \Leftrightarrow \psi'(\type_G(\mathbf w, v), \freq^{c'}_G(\mathbf w, v)) = 1).$$
Let $\gamma(w_1, \ldots, w_k)$ be the number (mod $q$) of $v$ such that $\varphi'(w_1, \ldots, w_k, v)$ is true.
Then for any good $G$ (doing arithmetic mod $q$), $$\gamma(w_1, \ldots, w_k)  = \modsum_{v\in V_G} \psi'(\type_G(\mathbf w, v), \freq^{c'}_G(\mathbf w, v)).$$
Grouping terms, we have
\begin{align*}
\gamma(w_1, \ldots, w_k) &= \modsum_{\tau' \in \Types(k+1)}\modsum_{f' \in \Z_q^{\Conn_{k+1}^{c'}}} \psi'(\tau', f') \cdot |\{ v \in V_G: \type_G(\mathbf w, v) = \tau') \wedge \freq_{G}(\mathbf w, v) = f'\}|\\
&= \modsum_{\tau', f'} \psi'(\tau', f')\cdot \lambda(\tau', f', \type_G(\mathbf w), \freq^{c}_{G}(\mathbf w))\\
&\mbox{ \quad\quad (applying Theorem~\ref{thm:extendcount}, and taking $c = (q-1)c'2^{(c')^2}$)}
\end{align*}
which is solely a function of $\type_G(\mathbf w)$ and
$\freq^{c}_G(\mathbf w)$. Thus, there is a function $\psi: \Types(k) \times \Z_q^{\Conn_k^{c}} \rightarrow \{0,1\}$ such that for all good $G$ and for all $w_1, \ldots, w_k \in V_G$,  $\psi(\type_G(\mathbf w), \freq_G^{c}(\mathbf w)) = 1$ if and only if $\gamma(\mathbf w) \equiv i\mod q$. Thus,
$$ \Pr_{G}[\forall w_1, \ldots, w_k, ((\Mod_q^i v, \varphi'(w_1, \ldots, w_k, v)) \Leftrightarrow \psi(\type_G(\mathbf w), \freq^{c}_G(\mathbf w)) = 1) ] \geq 1 -2^{-\Omega(n)},$$
as desired.

{\bf Case $\exists$:}
Suppose $\varphi(\alpha_1, \ldots, \alpha_k) = \exists \beta, \varphi'(\alpha_1, \ldots, \alpha_k, \beta)$. Let $c' = c(k+1, t-1, q)$ and
let $\psi' : \Types(k+1) \times \Z_q^{\Conn_{k+1}^{c'}} \rightarrow \{0, 1\}$ be such that
\begin{equation}
\label{eqexistscase}
 \Pr_{G}[\forall w_1, \ldots, w_k, v \in V_G, (\varphi'(w_1, \ldots, w_k, v) \Leftrightarrow \psi'(\type_G(\mathbf w, v), \freq_G(\mathbf w, v)) = 1) ] \geq 1 -2^{-\Omega(n)}.
\end{equation}

For this case, we may choose $c$ to be any integer at least $c'$.
Define $\psi: \Types(k) \times \Z_q^{\Conn_{k}^{c}} \rightarrow \{0,1\}$ by the rule:
$\psi(\tau, f) = 1$ if there is a $(\tau', f') \in \Types(k+1) \times \Freq_n(\tau', [k+1], c')$ extending
$(\tau, f)$ such that $\psi'(\tau', f') = 1$. 

Fix any $\mathbf w \in [n]^k$. We will show that
\begin{equation}
\label{eqexiststoshow}
\Pr_G[ (\exists v, \psi' (\type_G(\mathbf w, v), \freq_G^{c'}(\mathbf w, v)) = 1) \Leftrightarrow \psi(\type_G(\mathbf w), \freq_G^{c}(\mathbf w)) = 1] \geq 1 - 2^{-\Omega(n)}.
\end{equation}
Taking a union bound of \eqref{eqexiststoshow} over all $\mathbf w \in [n]^k$, and using Equation~\eqref{eqexistscase}, we conclude that
$$\Pr_{G \in G(n,p)}[\forall w_1, \ldots, w_k \in V_G, (\varphi(w_1, \ldots, w_k) \Leftrightarrow \psi(\type_G(\mathbf w), \freq_G(\mathbf w)) = 1)] \geq 1 - 2^{-\Omega(n)},$$
as desired.

It remains to show Equation~\eqref{eqexiststoshow}.
It will help to expose the edges of the random graph $G$ in three stages.

In the first stage, we expose all the edges between the vertices in $\{w_1, \ldots, w_k\}$.

For the second stage, let $s = \rho(c, q, p) \cdot n$ (where $\rho$ comes from Theorem~\ref{thm:superindep}) and pick distinct vertices $u_1, \ldots, u_s \in [n]\setminus\{w_1, \ldots, w_k\}$.
In the second stage, we expose all the unexposed edges between the vertices in $\{w_1, \ldots, w_k, u_1, \ldots, u_s\}$
(i.e., the edges between $u_i$s and $w_j$s, as well as the edges between the $u_i$s and $u_j$s).
Denote the resulting graph induced on $\{w_1, \ldots, w_k, u_1, \ldots, u_s\}$ after the second stage by
$A$ (so that $V_A = \{w_1, \ldots w_k, u_1, \ldots, u_s\}$).

In the third stage, we expose the rest of the edges in $G$.
Thus $G$ is sampled from $G(n, p \mid V_A, E_A)$.

Let $\tau$ denote the random variable $\type_G(\mathbf w)$.
Note that $\tau$ is determined after the first stage.
Let $\tau_1, \ldots, \tau_s$ denote the random variables $\type_G(\mathbf w, u_1),
\ldots, \type_G(\mathbf w, u_s)$. Note that $\tau_1, \ldots, \tau_s$ are
all determined after the second stage. Let $f$ denote the random variable $\freq_G(\mathbf w)$.
Let $f_1, \ldots, f_s$ denote the random variables $\freq_G(\mathbf w, u_1),
\ldots, \freq_G(\mathbf w, u_s)$. The variables $f, f_1, \ldots, f_s$ are all determined
after the third stage. Notice that the content of Theorem~\ref{thm:superindep}
is precisely a description of the distribution of $(f, f_1, \ldots, f_s)$.

We identify two bad events $B_1$ and $B_2$.

$B_1$ is defined to be the event: there exists $\sigma \in \Types(k+1)$ extending $\tau$,
with $\{k+1\} \in \Pi_{\sigma}$ (ie, types $\sigma$ where vertex $k+1$ is distinct from the
other vertices), such that $$|\{ i \in [s] : \tau_i = \sigma \}| \leq \frac12s\min\{p^k, (1-p)^k\}.$$
(This can be interpreted as saying that the type $\sigma$ appears abnormally infrequently amongst
the $\tau_i$).
Note that for any $\sigma$ extending $\tau$, the events ``$\tau_i = \sigma$", for $i \in [s]$,
are independent conditioned on the outcome of the first stage, since they depend on disjoint
sets of edges of $G$.
Also, for each $i$ and each $\sigma$ extending $\tau$ with $\{k+1\} \in \Pi_{\sigma}$, the probability that $\tau_i = \sigma$ is $\geq \min\{p^k, (1-p)^k\}$.
Therefore, applying the Chernoff bound, and taking a union bound over all $\sigma$ extending $\tau$ with $\{k+1\} \in \Pi_{\sigma}$, we see that
$$\Pr[B_1] \leq 2^k \exp(-s \min\{p^k,(1-p)^k\}) \leq 2^{-\Omega(n)}.$$

Now let
\begin{align*}
S = \{(\sigma, g) \in &\  \Types(k+1) \times \Freq_n(\sigma, [k+1], c') \mid \{k+1\} \in \Pi_{\sigma} \\
&\mbox{ AND } (\sigma,g) \mbox{ extends } (\tau, f)
 \mbox{ AND } \psi'(\sigma, g) = 1\}.
\end{align*}
$B_2$ is defined to be the event: $S \neq \emptyset$ and for each $i \in [s]$, $(\tau_i, f_i) \not\in S$.
We study the probability of $\neg B_1 \wedge B_2$.
Let $U$ be the set of $(d,  d_1, \ldots, d_s) \in \Freq_n(\tau, [k], c) \times \prod_{i}\Freq_n(\tau_i, [k+1], c')$ such that
\begin{enumerate}
\item The set $S(d)$ defined by
\begin{align*}
S(d) = \{(\sigma, g) \in &\ \Types(k+1) \times \Freq_n(\sigma, [k+1], c') \mid \{k+1\} \in \Pi_{\sigma}\\
&\mbox{ AND } (\sigma,g) \mbox{ extends } (\tau, d) \mbox{ AND } \psi'(\sigma, g) = 1\},
\end{align*}
is nonempty.
\item For each $i \in [s]$, $(\tau_i, d_i) \not\in S(d)$.
\end{enumerate}
By definition, the event $B_2$ occurs precisely when $(f, f_1, \ldots, f_s) \in U$.

By Theorem~\ref{thm:superindep}, for any fixing of $E_A$, the probability that $(f, f_1, \ldots, f_s) \in U$ is at most $2^{-\Omega(n)}$ more than the probability that $(h, h_1, \ldots, h_s) \in U$. As the event $B_1$ is solely a function
of $E_A$, we conclude that $\Pr[\neg B_1 \wedge (f, f_1, \ldots, f_s) \in U)] \leq \Pr[\neg B_1 \wedge(h, h_1, \ldots, h_s) \in U] + 2^{-\Omega(n)}$.

It remains to bound $\Pr[\neg B_1 \wedge(h, h_1, \ldots, h_s) \in U]$. If $S(h) \neq \emptyset$, take a
$(\sigma, g) \in S(h)$. In the absence of $B_1$, the number of $i\in[s]$ with $\tau_i = \sigma$ is at least
$\frac12 s \min\{p^k, (1-p)^k\}$. For all these $i$, it must hold that $h_i \neq g$ in order for
$(h, h_1, \ldots, h_s)$ to lie in $U$. Therefore,
$$\Pr[\neg B_1 \wedge (h, h_1, \ldots, h_s)\in U] \leq \left(1 - \frac{1}{|\Freq_n(\tau, k+1, c')|}\right)^{\frac12 s \min\{p^k, (1-p)^k\}}.$$
Notice that this last quantity is of the form $2^{-\Omega_{p, q, k, d}(s)}$.

Putting everything together,
$$\Pr[\neg B_1 \wedge B_2] \leq \Pr[\neg B_1 \wedge (h, h_1, \ldots, h_s) \in U] + 2^{-\Omega(n)} \leq  2^{-\Omega(s)} + 2^{-\Omega(n)} \leq 2^{-\Omega(n)}.$$

Therefore, with probability at least $1 - 2^{-\Omega(n)}$, the event $B_2$ does not occur.
The next claim finishes the proof of Equation~\eqref{eqexiststoshow}, and with that the proof of
Theorem~\ref{thm:foptopoly}.

\begin{claim}
If $B_2$ does not occur, then
$$(\exists v, \psi' (\type_G(\mathbf w, v), \freq_G^{c'}(\mathbf w, v)) = 1) \Leftrightarrow (\psi(\type_G(\mathbf w), \freq_G^{c}(\mathbf w)) = 1).$$
\end{claim}
\begin{proof}
Let $\tau = \type_G(\mathbf w)$ and $f = \freq_G^{c}(\mathbf w)$.

If $\psi(\tau, f) = 0$, then we know that for all $(\tau', f') \in \Types(k+1) \times \Freq_n(\tau', k+1 , c')$ extending
$(\tau, f)$, we have $\psi'(\tau', f') = 0$. Thus by Lemma~\ref{lem:ffreqcrux}, for all $v \in V_G$,
$ \psi' (\type_G(\mathbf w, v), \freq_G^{c'}(\mathbf w, v)) = 0$, as required.

If $\psi(\tau, f) = 1$, then we consider two situations.
\begin{itemize}
\item {\bf The self-fulfilling situation:} If there is a $(\tau', f') \in \Types(k+1) \times \Freq_n(\tau', k+1, c')$ extending $(\tau, f)$ with
$\{k + 1\} \not\in \Pi_{\tau'}$ and $\psi'(\tau', f') = 1$. In this case, take any $j \in [k]$
with $\Pi_{\tau'}(j) = \Pi_{\tau'}(k+1)$, and let $v = w_j$.
Thus $\type_G(\mathbf w, v)= \tau'$.
By Lemma~\ref{lem:self-fulfill}, since $(\tau', f')$ extends $(\tau, f)$ with $\{k+1\} \not\in \Pi_{\tau'}$,
it follows that $\freq_G^{c'}(\mathbf w, v) = f'$.
Therefore, with this choice of $v$, we have
$\psi'(\type_G(\mathbf w, v), \freq_{G}^{c'}(\mathbf w, v)) = 1$, as required.
\item {\bf The default situation:} In this case, there is a $(\tau', f') \in \Types(k+1) \times \Freq_n(\tau', k+1, c')$ extending $(\tau, f)$
with $\{k+1\} \in \Pi_{\tau'}$ and  $\psi'(\tau', f') = 1$. This is precisely the statement that
$S \neq \emptyset$. Therefore, by the absence of the event $B_2$, there must be an $i \in [r]$
such that $(\tau_i, f_i) \in S$. Taking $v = u_i$, we see that
$\psi'(\type_G(\mathbf w, v), \freq_{G}^{c'}(\mathbf w, v)) = 1$, as required.
\end{itemize}
This completes the proof of the claim.
\end{proof}
\end{proof}

\section{ Counting Extensions}
\label{sec:extendcount}
In this section we prove Theorem~\ref{thm:extendcount}. 

\subsection{Subgraph frequency arithmetic}
We begin with a definition.
A {\em partial matching} between two
$I$-labelled graphs $F_1, F_2$ is a subset $\eta \subseteq (V_{F_1}\setminus \mathcal L(F_1)) \times (V_{F_2} \setminus \mathcal L(F_2))$ that is one-to-one.
For two graphs $F_1, F_2$, let $\pmat(F_1, F_2)$ be the set of all partial matchings between them.

\begin{definition}[Gluing along a partial matching] Let $F_1$ and $F_2$ be two $I$-labelled graphs, and let $\eta \in \pmat(F_1, F_2)$.
Define {\em the gluing of $F_1$ and $F_2$ along $\eta$}, denoted $F_1 \vee_{\eta} F_2$, to be the graph obtained by first
taking the disjoint union of $F_1$ and $F_2$, identifying pairs of vertices with the same label, and then
identifying the vertices in each pair of $\eta$ (and removing duplicate edges). We omit the
subscript when $\eta = \emptyset$.
\end{definition}

We have the following simple identity.

\begin{lemma}
\label{lem:productinj}
For any $I$-labelled graphs $F_1, F_2$, any graph $G$ and any $\mathbf w \in V_G^I$:
\begin{equation}
\label{kveeidentity}
[F_1](G, \mathbf w)\cdot[F_2](G, \mathbf w) = \sum_{ \eta \in \pmat(F_1, F_2)} [F_1 \vee_\eta F_2](G, \mathbf w).
\end{equation}
\end{lemma}
\begin{proof}
We give a bijection $$\alpha : \Inj(F_1, (G,\mathbf w)) \times \Inj(F_2, (G, \mathbf w)) \rightarrow \coprod_{\eta \in \pmat(F_1, F_2)} \Inj( F_1 \vee_{\eta} F_2, (G, \mathbf w)).$$
Define $\alpha(\chi_1, \chi_2)$ as follows. Let $\eta = \{ (v_1, v_2) \in (V_{F_1}\setminus \mathcal L(F_1)) \times (V_{F_2} \setminus \mathcal L(F_2)) \mid \chi_1(v_1) = \chi_2(v_2) \}$.
Let $\iota_1 \in \Inj(F_1, F_1 \vee_\eta F_2)$ and $\iota_2 \in \Inj(F_2, F_1 \vee_\eta F_2)$ be the natural inclusions.
Let $\chi \in \Inj(F_1 \vee_\eta F_2, (G, \mathbf w))$ be the unique homomorphism such that
for all $v \in V_{F_1}$, $\chi \circ \iota_1 (v) = \chi_1(v)$, and for all $v\in V_{F_2}$, 
$\chi \circ \iota_2(v) = \chi_2(v)$.
We define $\alpha(\chi_1, \chi_2) := \chi$.

To see that $\alpha$ is a bijection, we give its inverse $\beta$. Let $\eta \in \pmat(F_1, F_2)$ and $\chi \in \Inj(F_1 \vee_\eta F_2, (G, \mathbf w))$. 
Let $\iota_1 \in \Inj(F_1, F_1 \vee_\eta F_2)$ and $\iota_2 \in \Inj(F_2, F_1 \vee_\eta F_2)$ be the natural inclusions.
Define $\beta(\chi) := (\chi \circ \iota_1, \chi \circ \iota_2)$.

Then $\beta$ is the inverse of $\alpha$.
\end{proof}

We can now prove Lemma~\ref{lem:kgraphtoconnected}.

{\bf Lemma~\ref{lem:kgraphtoconnected} (Label-connected subgraph frequencies determine all subgraph frequencies, restated)}
{\it
For every $k$-labelled graph $F'$ with $|V_{F'} \setminus \mathcal L(F')| \leq t$, there is a polynomial $\delta_{F'} \in \mathbb Z[(X_F)_{F \in \Conn_k^{t}}]$ such that for all graphs $G$ and $\mathbf w \in V_G^k$,
$$[F'](G, \mathbf w) = \delta_{F'}(x),$$
where $x \in \mathbb Z^{\Conn_k^{t}}$ is given by $x_{F} = [F](G, \mathbf w)$.
}

\begin{proof}
By induction on the number of connected components of $F'\setminus \mathcal L(F')$.
If $F'$ is \konn, then we take $\delta_{F'}(\mathbf X) = X_{F'}$.

Now suppose $F'$ is label-disconnected. Write $F' = F_1 \vee F_2$ where $F_1$ and $F_2$ are both $k$-labelled graphs, and $F_1\setminus \mathcal L(F_1)$ and $F_2\setminus \mathcal L(F_2)$ have fewer connected components.

By equation~\eqref{kveeidentity}, for all $G$ and $\mathbf w$,
\begin{align*}
[F_1 \vee F_2](G, \mathbf w) = [F_1](G, \mathbf w)\cdot[F_2](G, \mathbf w) - \sum_{\emptyset \neq \eta \in \pmat(F_1, F_2)} [F_1 \vee_\eta F_2](G, \mathbf w).
\end{align*}
Observe that for any $\eta \neq \emptyset$, each graph $F_1 \vee_\eta F_2$ has at least one fewer \konn  component than $F_1 \vee F_2 = F'$. Thus, by induction hypothesis, we may take
$$\delta_{F'}(\mathbf X) = \delta_{F_1}(\mathbf X) \cdot \delta_{F_2}(\mathbf X) - \sum_{\emptyset \neq\eta \in \pmat(F_1, F_2)} \delta_{F_1 \vee_\eta F_2}(\mathbf X).$$
This completes the proof of the lemma.
\end{proof}


\subsection{Proof of Theorem~\ref{thm:extendcount}}

{\bf Theorem~\ref{thm:extendcount} (restated)}
{\it
Let $q$ be a prime, let $k, b >0 $ be integers and let $a \geq (q-1)\cdot b \cdot|\Conn_{k+1}^{b}|$.
There is a function
$$\lambda: \Types(k+1) \times \Z_q^{\Conn_{k+1}^{b}} \times \Types(k) \times \Z_q^{\Conn_k^{a}} \rightarrow \Z_q$$
such that for  all $\tau' \in \Types(k+1)$, $f' \in \Z_q^{\Conn_{k+1}^{b}}$,  $\tau \in \Types(k)$, $f \in \Z_q^{\Conn_k^a}$, it holds that for every graph $G$, and every $w_1, \ldots, w_k \in V_G$ with $\type_G(\mathbf w) = \tau$ and $\freq_G^{a}(\mathbf w) = f$, the cardinality of the set
$$\{v \in V_G : \type_G(\mathbf w, v) = \tau' \wedge \freq_G^{b}(\mathbf w, v) = f'\}$$
is congruent to $\lambda(\tau', f', \tau, f) \mod q.$
}

\begin{proof}
We describe the function $\lambda(\tau', f', \tau, f)$ explicitly.
If $\tau'$ does not extend $\tau$, then we set $\lambda(\tau', f', \tau, f) = 0$.

Now assume $\tau'$ extends $\tau$.
We take cases on whether $k+1$ is a singleton in $\Pi_{\tau'}$ or not.

{\bf Case 1:} $\{k+1\} \in \Pi_{\tau'}$. In this case, there is an $I \subseteq [k]$
such that $\type_G(w_1, \ldots, w_k, v) = \tau'$ if and only if $v \not\in \{w_1, \ldots, w_k\}$
and $(v, w_i) \in E_G\  \Leftrightarrow\  i \in I$ (explicitly, $I = \{ i \in [k] \mid \{\{k+1\}, \Pi_{\tau'}(i)\} \in E_{\tau'} \}$).

Let for each $u,v \in V_G$, let $x_{uv} \in \{0,1\}$, where
$x_{uv} = 1$ if and only if $u$ is adjacent to $v$ in $G$.

Then, using the fact that $q$ is prime, the number (mod $q$) of $v$ with $\type_G(\mathbf w, v) = \tau'$ and $\freq_G(\mathbf w, v) = f'$ can be compactly expressed as (doing arithmetic mod $q$):
\begin{align*}
&\modsum_{v \in V_G\setminus \{w_1, \ldots, w_k\}} \prod_{i \in I} x_{vw_i} \prod_{j \in [k]\setminus I} (1- x_{vw_j})
\prod_{F \in \Conn_{k+1}^b} \left(1 - \left( [F]_q(G, \mathbf w, v)  - f'_F\right)^{q-1}\right)
\end{align*}
Expanding, the expression $\prod_{i \in I} x_{vw_i} \prod_{j \in [k]\setminus I} (1- x_{vw_j})$ may be expressed in
the form $\sum_{S \subseteq [k]} b_S \prod_{i \in S} x_{vw_i}$.
Using Lemma~\ref{lem:productinj}, the expression
$\prod_{F \in \Conn_{k+1}^b} \left(1 - \left( [F]_q(G, \mathbf w, v)  - f'_F\right)^{q-1}\right)$ may be
expressed in the form $\sum_{j} c_j [F_j]_q(G, \mathbf w, v)$, where each $F_j$ is a $k+1$-labelled graph with
at most $|\Conn_{k+1}^b|\cdot b \cdot (q-1) \leq a$ vertices.

Thus we may rewrite the expression for $\lambda(\tau', f', \tau, f)$ as:
\begin{align*}
&\sum_{v \in [n]\setminus\{w_1, \ldots, w_k\}}\left(\sum_{S} b_S\prod_{i \in S} x_{vw_i} \right) \left(\sum_{j} c_j [F_j]_q(G, \mathbf w, v) \right)\\
&= \sum_{S, j} b_S c_j \sum_{v\in [n]\setminus\{w_1, \ldots, w_k\}} \left(\left(\prod_{i\in S} x_{vw_i}\right) [F_j]_q(G, \mathbf w, v)\right)\\
&= \sum_{S, j} b_S c_j [F'_{S,j}]_q(G, \mathbf w),
\end{align*}
where $F'_{S,j}$ is the $k$-labelled graph obtained from $F_j$ by
\begin{enumerate}
\item[(a)] For each $i \in S$, adding an edge between the vertex labelled $k+1$ and the vertex labelled $i$, and
\item[(b)] Removing the label from the vertex labelled $k+1$.
\end{enumerate}

Finally, note that by Lemma~\ref{lem:kgraphtoconnected}, $[F'_{S,j}]_q(G, \mathbf w)$ is determined by $\freq^a_G(\mathbf w)$.

{\bf Case 2:} $\{k+1\}\not\in \Pi_{\tau'}$. This case is much
easier to handle. Pick any $j \in [k]$ such that $\Pi_{\tau'}(j) =
\Pi_{\tau'}(k+1)$. Then there is only one $v \in V_G$ such that
$\type_G(\mathbf w, v) = \tau'$ (namely, $w_j$).

Then $\lambda(\tau', f', \tau, f) = 1$ if and only if for all $F' \in \Conn_{k+1}^{b}$, $f'_{F'} = f_{F}$, where $F\in \Conn_{k}^{b}$ is the graph obtained by identifying the vertex labelled $k+1$ with the vertex labelled $j$, and labelling this new vertex $j$. Otherwise $\lambda(\tau', f', \tau, f) = 0$.

This completes the definition of our desired function $\lambda$.
\end{proof}

\section{The Distribution of Labelled Subgraph Frequencies mod q}
\label{sec:indep}

In this section, we prove Theorem~\ref{thm:superindep}. As in Section~\ref{sec:equi},
the proof will be via an intermediate theorem (Theorem~\ref{thm:indep-subgraphs}) that proves the equidistribution of the number
of {\em copies} of labelled subgraphs in $G(n,p)$.

\subsection{Equidistribution of labelled subgraph copies}

First, we gather some simple observations about injective homomorphisms from \konn graphs for later use (the proofs are simple and are omitted).
\begin{proposition}[Simple but delicate observations about \konn graphs]
\label{remark:konn}
Let $F, F' \in \Conn_I^t$.
Let $G$ be a graph and let $\mathbf w \in V_G^I$ with all
$(w_i)_{i \in I}$ distinct.
\begin{enumerate}
\item If $E \in \Cop(F, (G, \mathbf w))$, the $|E| = |E_F|$.
\item If $F \not\cong F'$, we have $\Cop(F, (G, \mathbf w)) \cap \Cop(F', (G, \mathbf w)) = \emptyset$.
\item Let $\chi_1, \ldots, \chi_r \in \Inj(F, (G, \mathbf w))$ be such that
for any distinct $j, j' \in [r]$, $\chi_j(V_F\setminus \mathcal L(F)) \cap \chi_{j'}(V_F\setminus \mathcal L(F)) = \emptyset$. Let
$\chi \in \Inj(F', (G,\mathbf w))$. Suppose $\chi(E_{F'}) \subseteq (\cup_j \chi_j(E_F))$. Then there is a $j \in [r]$ such that $\chi(E_{F'}) \subseteq \chi_j(E_F)$.
\end{enumerate}
\end{proposition}

We can now state and prove an equidistribution theorem for
the number of copies of labelled subgraphs in a conditioned random graph.
Theorem~\ref{thm:superindep} will follow from this.

\begin{theorem}
\label{thm:indep-subgraphs}
Let $A$ be a graph with $V_A \subseteq [n]$ and $|V_A| \leq n'$. Let $\mathbf w = (w_1, \ldots, w_k) \in V_A^k$
with $w_1, \ldots, w_k$ distinct.
Let $u_1, \ldots u_s \in V_A \setminus \{w_i: i \in I\}$ be distinct.
Let $F_1, \ldots, F_\ell$ be distinct $k$-labelled {\konn} graphs, with $1 \leq |E_{F_i}| \leq d$.
Let $H_1, \ldots, H_{\ell'}$ be distinct $k+1$-labelled {\konn} graphs dependent on label $k+1$,
with $1 \leq |E_{H_i}| \leq d$.

Let $G \in G(n,p \mid V_A, E_A)$. Then the distribution of
$$((\langle F_i\rangle_q(G, \mathbf w))_{i \in [\ell]}, (\langle H_{i'}\rangle_q(G, \mathbf w, u_{j'}))_{i' \in [\ell'], j' \in [s]})$$ on $\Z_q^{\ell + s\ell'}$ is
$2^{-\Omega_{q,p,d}(n-n') + (\ell + \ell's)\log q}$-close to uniform in statistical distance.
\end{theorem}
\begin{proof}
By the Vazirani XOR lemma (Lemma~\ref{lem:xor}), it suffices to show that for any nonzero $(c, c') \in \Z_q^\ell\times \Z_q^{\ell'\times s}$,
we have
$\left|\E\left[ \omega^R\right] \right| \leq 2^{-\Omega_{q,p,d}(n-n')}$, 
where $$R := \modsum_{i \in [\ell]}c_i \langle F_i\rangle_q(G, \mathbf w) + \modsum_{i' \in [\ell']}\modsum_{j' \in [s]} c'_{i'j'} \langle H_{i'} \rangle_q(G, \mathbf w, u_{j'})$$
and $\omega \in \mathbb C$ is a primitive $q^{\rm{th}}$-root of unity.

We will show this by appealing to Lemma~\ref{lem:polybias1}. Let
$m = {n \choose 2} - {a \choose 2}$.
Let $\mathbf z \in \{0,1\}^{[n] \choose 2}$ be the
random variable where, for each $e \in { [n] \choose 2}$, $z_e = 1$ if and only if edge $e$ is present in $G$. Thus,
independently for each $e \in {[n] \choose 2} \setminus {V_A \choose 2}$, $\Pr[z_e = 1] = p$, while for $e \in {V_A \choose 2}$, the value of $z_e$ is either identically 1 or identically 0 (depending on whether $e \in E_A$ or not).

We may now express
$R$ in terms of the $z_e$. We have,
\begin{align*}
R &= \modsum_{i \in [\ell]}c_i \langle F_i\rangle_q(G, \mathbf w) + \modsum_{i' \in [\ell']}\modsum_{j' \in [s]} c'_{i'j'} \langle H_{i'} \rangle_q(G, \mathbf w, u_{j'})\\
&= \modsum_{i \in [\ell]} c_i \sum_{E \in \Cop(F_i, (\K_n, \mathbf w))} \prod_{e \in E} z_e + \modsum_{i' \in [\ell']}\modsum_{j' \in [s]} c'_{i'j'} \sum_{E \in \Cop(H_{i'}, (\K_n, \mathbf w, u_{j'}))} \prod_{e \in E} z_e\\
&= \modsum_{E \in \mathcal F_1} c_E \prod_{e \in E} z_e + \modsum_{E \in \mathcal F_2} c'_E \prod_{e\in E} z_e,
\end{align*}
where $\mathcal F_1 \subseteq 2^{[n] \choose 2}$ is the set $\bigcup_{i \in [\ell]: c_i \neq 0} \Cop(F_i, (\K_n, \mathbf w))$, $\mathcal F_2$ is the set $\bigcup_{i' \in [\ell'], j' \in [s]: c'_{i'j'} \neq 0} \Cop(H_{i'}, (\K_n, \mathbf w, u_{j'}))$, for each $E \in \mathcal F_1$, $c_E = c_i$ where $i \in [\ell]$ is such that $E \in \Cop(F_i, (\K_n, \mathbf w))$ (note that by Proposition~\ref{remark:konn} there is exactly one such $i$), and similarly, for $E \in \mathcal F_2$, $c'_E = \sum_{i' \in [\ell'], j'\in [s]: E \in \Cop(H_{i'}, (\K_n, \mathbf w, u_{j'})} c'_{i'j'}$. Thus if $E$ is such that there is a unique $(i', j') \in [\ell']\times[s]$ for which $E \in \Cop(H_{i'}, (\K_n, \mathbf w, u_{j'}))$ and $c'_{i', j'} \neq 0$, then $c'_E \neq 0$.

Let $Q(\mathbf Z) \in \Z_q[\mathbf Z]$, where $\mathbf Z = (Z_e)_{e \in {[n]\choose 2} \setminus {V_A \choose 2}}$,  be the polynomial
$$\sum_{E \in \mathcal F_1} c_E\prod_{e \in E \cap {V_A \choose 2}} z_e \prod_{e \in E \setminus {V_A \choose 2}} Z_e + \sum_{E \in \mathcal F_2} c'_E\prod_{e \in E \cap {V_A \choose 2}} z_e \prod_{e \in E \setminus {V_A \choose 2}} Z_e.$$

Let $\mathbf{\widehat{z}} \in \{0,1\}^{{[n] \choose 2} \setminus {V_A \choose 2}}$ be the random variable $\mathbf z$ restricted to the coordinates indexed by ${[n]\choose 2}\setminus {V_A \choose 2}$ (thus each coordinate of $\mathbf{\widehat z}$ independently equals $1$ with probability $p$). Then $R = Q(\mathbf{\widehat{z}})$.
We wish to show that
\begin{equation}
\label{eq:indep-uniform}
\left| \E\left[ \omega^{Q(\mathbf{ \widehat z})}\right] \right| \leq 2^{-\Omega_{q,p,d}(n - n')}.
\end{equation}

We do this by demonstrating that the polynomial $Q(\mathbf Z)$ satisfies the hypotheses of Lemma~\ref{lem:polybias1}.

Let $d_1^* = \max_{i: c_i \neq 0} |E_{F_i}|$.
Let $d_2^* = \max_{i', j': c'_{i'j'} \neq 0} |E_{H_{i'}}|$.
We take cases depending on whether $d_1^* < d_2^*$ or $d_1^* \geq d_2^*$.

{\bf Case 1:} Suppose $d_1^* < d_2^*$.
Let $i_0', j_0'$ be such that $c'_{i_0'j_0'} \neq 0$ and $|E_{H_{i_0'}}|  = d_2^*$.
Then $Q(\mathbf Z)$ may be written as $ \sum_{E \in \mathcal F} c'_E \prod_{e \in E} Z_e + Q'(\mathbf Z)$, where $\mathcal F = \{ E \in \mathcal F_2:
E \cap {V_A \choose 2} = \emptyset \}$ and $\deg(Q') < d_2^*$.

Let $\chi_1, \chi_2, \ldots, \chi_r \in \Inj(H_{i_0'}, (\K_n, \mathbf w, u_{j_0'}))$ be a collection of homomorphisms such that:
\begin{enumerate}
\item For all $j \in [r]$, we have  $\chi_j(V_{H_{i_0'}} \setminus \mathcal L(H_{i_0'})) \subseteq [n] \setminus V_A$.
\item For all distinct $j, j' \in [r]$, we have $\chi_j(V_{H_{i_0'}}\setminus\mathcal L(H_{i_0'})) \cap \chi_{j'}(V_{H_{i_0'}}\setminus\mathcal L(H_{i_0'})) = \emptyset$.
\end{enumerate}
Such a collection can be chosen greedily so that
$r = \Omega(\frac{n-n'}{d})$.
Let $E_j \in \Cop(H_{i_0'}, (\K_n, \mathbf w, u_{j_0'}))$ be given by $\chi_j(E_{H_{i_0'}})$.
Let $\mathcal E$ be the family of sets $\{E_1, \ldots, E_r\} \subseteq \mathcal F$. We observe the following properties of the $E_j$:
\begin{enumerate}
\item For each $j \in [r]$, $|E_j| = d_2^*$ (since $\chi_j$ is injective and $w_1, \ldots, w_k, u_{j_0'}$ are distinct).
\item For each $j \in [r]$, $c'_{E_j} \neq 0$. This is because there is a unique $(i', j')$ (namely $(i_0', j_0')$) for which $c_{i'j'} \neq 0$ and $E_j \in \Cop(H_{i'}, (\K_n, \mathbf w, u_{j'}))$. Indeed, if $j' \neq j_0'$, then each $E^* \in
\Cop(H_{i'}, (\K_n, \mathbf w, u_{j'}))$ has some element incident on $u_{j'}$ (while $E_j$ does not). On the other hand, if $j' = j_0'$ and $i' \neq i_0'$, then Proposition \ref{remark:konn} implies that $\Cop(H_{i'}, (\K_n, \mathbf w, u_{j'})) \cap \Cop(H_{i_0'}, (\K_n, \mathbf w, u_{j_0'})) = \emptyset$.
\item For distinct $j, j' \in [r]$, $E_j \cap E_{j'} = \emptyset$ (by choice of the $\chi_j$).
\item For any $S \in \mathcal F \setminus \mathcal E$,
$|S \cap (\cup_j E_j)| < d_2^*$.
To see this, take any $S \in \mathcal F \setminus \mathcal E$ and suppose $|S \cap (\cup_j E_j)| \geq d_2^*$.
Let $i' \in [\ell'], j' \in [s]$ be such that $S \in \Cop(H_{i'}, (\K_n, \mathbf w, u_{j'}))$. Let $\chi \in \Inj(H_{i'}, (\K_n, \mathbf w, u_{j'}))$
with $\chi(E_{H_{i'}}) = S$. By choice of $d_2^*$,
we know that $|S| \leq d_2^*$.
Therefore, the only way that $|S \cap (\cup_j E_j)|$ can be $\geq d_2^*$ is
if (a) $|S| = d_2^*$, and (b) $S \cap (\cup_j E_j) = S$, or in other words, $S \subseteq (\cup_j E_j)$. Since $H_{i'}$ is dependent on label $k+1$, we know
that $S$ has some element incident on vertex $u_{j'}$, and thus (b) forces $j' = j_0'$ (otherwise no $E_j$ is incident
on $u_{j'}$). Now by Proposition~\ref{remark:konn}, this implies that $S \subseteq E_j$ for some $j$. But since
$|E_j| = |S|$, we have $S = E_j$, contradicting our choice of $S$.
Therefore, $|S \cap (\cup_j E_j)| < d_2^*$ for any $S \in \mathcal F \setminus \mathcal E$.
\end{enumerate}

It now follows that
 $Q(\mathbf Z), \mathcal F$ and $\mathcal E$ satisfy the hypothesis of
Lemma~\ref{lem:polybias1}. Consequently, (noting that $d_2^*\leq d$) Equation~\eqref{eq:indep-uniform} follows, completing the proof in Case 1.

{\bf Case 2:} Suppose $d_1^* \geq d_2^*$.
Let $i_0$ be such that $c_{i_0} \neq 0$ and $|E_{F_{i_0}}|  = d_1^*$.
Then $Q(\mathbf Z)$ may be written as $ \sum_{E \in \mathcal F} (c_E + c'_E) \prod_{e \in E} Z_e + Q'(\mathbf Z)$, where $\mathcal F = \{ E \in \mathcal F_1 \cup F_2:
E \cap {V_A \choose 2} = \emptyset \}$ and $\deg(Q') < d_1^*$.

Let $\chi_1, \chi_2, \ldots, \chi_r \in \Inj(F_{i_0}, (\K_n, \mathbf w))$ be a collection of homomorphisms such that:
\begin{enumerate}
\item For all $j \in [r]$, we have  $\chi_j(V_{F_{i_0}} \setminus \mathcal L(F_{i_0})) \subseteq [n] \setminus V_A$.
\item For all distinct $j, j' \in [r]$, we have $\chi_j(V_{F_{i_0}}\setminus\mathcal L(F_{i_0})) \cap \chi_{j'}(V_{F_{i_0}}\setminus\mathcal L(F_{i_0})) = \emptyset$.
\end{enumerate}
Such a collection can be chosen greedily so that
$r = \Omega(\frac{n-n'}{d})$.
Let $E_j \in \Cop(F_{i_0}, (\K_n, \mathbf w))$ be given by $\chi_j(E_{F_{i_0}})$.
Let $\mathcal E$ be the family of sets $\{E_1, \ldots, E_r\} \subseteq \mathcal F$. We observe the following properties of the $E_j$:
\begin{enumerate}
\item For each $j \in [r]$, $|E_j| = d_1^*$ (since $\chi_j$ is injective and $w_1, \ldots, w_k$ are distinct).
\item For each $j \in [r]$, $c_{E_j} + c'_{E_j} \neq 0$. This is because $c_{E_j} = c_{i_0} \neq 0$ and for any $(i', j')$, $E_j \not\in \Cop(H_{i'}, (\K_n, \mathbf w, u_{j'}))$ (and so $c'_{E_j} = 0$). To see the latter claim, note that each $E^* \in \Cop(H_{i'}, (\K_n, \mathbf w, u_{j'}))$ has an element incident on $u_{j'}$ (which $E_j$ does not).
\item For distinct $j, j' \in [r]$, $E_j \cap E_{j'} = \emptyset$ (by choice of the $\chi_j$).
\item For any $S \in \mathcal F \setminus \mathcal E$,
$|S \cap (\cup_j E_j)| < d_1^*$.
To see this, take any $S \in \mathcal F \setminus \mathcal E$ and suppose $|S \cap (\cup_j E_j)| \geq d_1^*$.
\begin{enumerate}
\item If $S \in \mathcal F_1$, then let $i \in [\ell]$ be such that
$S \in \Cop(F_{i}, (\K_n, \mathbf w))$. Let $\chi \in \Inj(F_{i}, (\K_n, \mathbf w))$
with $\chi(E_{F_{i}}) = S$. We know that $|S| \leq d_1^*$.
Therefore, the only way that $|S \cap (\cup_j E_j)|$ can be $\geq d_1^*$ is
if (1) $|S| = d_1^*$, and (2) $S \cap (\cup_j E_j) = S$, or in other words, $S \subseteq (\cup_j E_j)$.
However, by Proposition~\ref{remark:konn}, this implies that $S \subseteq E_j$ for some $j$. But since
$|E_j| = |S|$, we have $S = E_j$, contradicting our choice of $S$.
\item If $S \in \mathcal F_2$, then let $i' \in [\ell'], j' \in [s]$ be such that $S \in \Cop(H_{i'}, (\K_n, \mathbf w, u_{j'}))$. Let $\chi \in \Inj(H_{i'}, (\K_n, \mathbf w, u_{j'}))$
with $\chi(E_{H_{i'}}) = S$. We know that $|S| \leq d_2^* \leq d_1^*$.
Now $S$ has an element incident on $u_{j'}$. On the other hand
none of the $E_j$ have any edges incident on $u_{j'}$. Therefore
$|S \cap ( \cup_j E_j)| < |S| \leq d_1^*$.
\end{enumerate}
Therefore,
$|S \cap (\cup_j E_j)| < d_1^*$ for any $S \in \mathcal F \setminus \mathcal E$.
\end{enumerate}
It now follows that
 $Q(\mathbf Z), \mathcal F$ and $\mathcal E$ satisfy the hypothesis of
Lemma~\ref{lem:polybias1}. Consequently, (noting that $d_1^* \leq d$) Equation~\eqref{eq:indep-uniform} follows, completing the proof in Case 2.
\end{proof}

\subsection{Proof of Theorem~\ref{thm:superindep}}

{\bf Theorem~\ref{thm:superindep} (restated) }
{\it
Let $a \geq b$ be positive integers.
Let $A$ be a graph with $V_A \subseteq [n]$ and $|V_A| \leq n' \leq n/2$. Let $G \in G(n,p \mid V_A, E_A)$. Let $\mathbf w = (w_1, \ldots, w_k) \in V_A^k$, and let $u_1, \ldots, u_s \in V_A\setminus\{w_1, \ldots, w_k\}$ be distinct. Let $\tau = \type_G(\mathbf w)$ and let $\tau_i = \type_G(\mathbf w, u_i)$ (note
that $\tau, \tau_1,\ldots, \tau_s$ are already determined by $E_A$).
Let $f$ denote the random variable $\freq_{G}^a(\mathbf w)$.
Let $f_i$ denote the random variable $\freq_{G}^b(\mathbf w, u_i)$.

Then, there exists a constant $\rho = \rho(a, q, p) > 0$, such that if
$s \leq \rho \cdot n$, then the distribution of $(f, f_1, \ldots, f_s)$ over $\Freq_n(\tau, [k], a) \times \prod_{i} \Freq_n(\tau_i, [k+1], b)$
is $2^{-\Omega(n)}$-close to the distribution of $(h, h_1, \ldots, h_s)$ generated as follows:
\begin{enumerate}
\item $h$ is picked uniformly at random from $\Freq_n(\tau, [k], a)$.
\item For each $i$, each $h_i$ is picked independently and uniformly from the set of all $f' \in \Freq_n(\tau_i, [k+1], b)$
such that $(\tau_i, f')$ extends $(\tau, h)$.
\end{enumerate}
}

\begin{proof}
Let $\mathbf v = \mathbf w/\Pi_{\tau}$.
Let $F_1, \ldots, F_\ell$ be an enumeration of the elements of $\Conn_{\Pi_{\tau}}^a$.

Let $\Pi' \in \Partitions([k+1])$ equal $\Pi_{\tau} \cup \{ \{ k+1 \} \}$. Notice that
for each $i \in [s]$, $\Pi_{\tau_i} = \Pi'$.
Let $H_1, \ldots, H_{\ell'}$ be an enumeration of those elements of $\Conn_{\Pi'}^b$ that are
dependent on label $i^*$.

By Theorem~\ref{thm:indep-subgraphs} and the hypothesis on $s$ for a suitable constant $\rho$, the distribution of
$$(g, g^1, \ldots, g^s) = ((\langle F_i\rangle_q(G, \mathbf v))_{i \in [\ell]}, (\langle H_{i'} \rangle(G,\mathbf v, u_{j'}))_{i' \in [\ell'], j' \in [s]})$$ is $2^{-\Omega(n)}$ close to uniform over $\Z_q^{\ell + \ell's}.$
Given the vector $(g, g^1, \ldots, g^s)$, we may compute the vector $(f, f_1, \ldots, f_s)$ as follows:
\begin{enumerate}
\item For $F = \K_1([k])$, we have $f_{F} = n - |\Pi_{\tau}|$.
\item For all other $F \in \Conn_k^a$, let $i \in [\ell]$ be such that $F/\Pi_{\tau} \cong F_i$. Then $f_{F} = g_{i} \cdot \aut(F_i)$.
\item For $H \in \Conn_{k+1}^b$ dependent on label $k+1$, let $i' \in [\ell']$ be such that $H/(\Pi') \cong H_i'$. Then for each $j' \in [s]$, $(f_{j'})_{H} = g^{j'}_{i'} \cdot \aut(H_{i'})$.
\item For $H \in \Conn_{k+1}^b$ not dependent on label $k+1$ and for any $j' \in [s]$, there is a {\em unique} setting
of $(f_{j'})_H$ (given the settings above) that is consistent with the fact that $(\tau_j, f_j)$ extends $(\tau, f)$.
This follows from Lemma~\ref{lem:unique-extend}.
%
\end{enumerate}

This implies the desired claim about the distribution of $(f, f_1, \ldots, f_s)$.
\end{proof}

\section{Concluding Remarks}

The results presented here constitute the first systematic investigation
of the asymptotic probabilities of properties expressible in first-order logic
with counting quantifiers. Moreover, these results have been established
by combining, for the first time, 
algebraic methods related to multivariate polynomials over finite fields  with the method of quantifier elimination from mathematical
logic.

We conclude with two open problems:
\begin{enumerate}
\item What is the complexity of computing the numbers $a_0, \ldots, a_{q-1}$
in Theorem~\ref{thm:Main}? We know that it is $\PSPACE$-hard to compute these numbers (it is already
$\PSPACE$-hard to tell if the asymptotic probability of a $\FO$ sentence is $0$ or $1$). Our proof shows that they may be computed in time
 $2^{2^{2^{\ldots}}}$ of height proportional to the quantifier depth of the
formula. It is likely that a more careful analysis of our approximation
of $\FOM$ by polynomials can yield better upper bounds.
\item Is there a modular convergence law for $\FO[\Mod_m]$ for arbitrary $m$?
The same obstacles that prevent the Razborov-Smolensky approach from
generalizing to $\mathsf{AC0}[\Mod_6]$ impede us.
Perhaps an answer to the above question
will give some hints for $\mathsf{AC0}[\Mod_6]$?
\end{enumerate}

\section*{Acknowledgements}
Swastik Kopparty is very grateful to Eli Ben-Sasson, Danny Gutfreund and
Alex Samorodnitsky for encouragement and stimulating discussions.
We would also like to thank Miki Ajtai, Ron Fagin, Prasad Raghavendra,
Ben Rossman, Shubhangi Saraf and Madhu Sudan for valuable discussions.

\end{document}